\documentclass{article}
\usepackage[a4paper,margin=0.81in,top=3cm,bottom=3.5cm,footskip=0.5cm]{geometry}
\usepackage[a4paper]{geometry}
\usepackage[utf8]{inputenc}
\usepackage{amsmath,amsfonts,amssymb,amsthm,bbm,mathrsfs}
\usepackage[shortlabels]{enumitem}
\usepackage{xcolor}
\usepackage{graphicx}
\usepackage{tikz}
\usepackage{float}
\usepackage{appendix}
\usepackage{setspace}
\usepackage{fancyhdr}
\usepackage{hyperref}
\hypersetup{
    colorlinks=true,
    linkcolor=black,
    filecolor=magenta,      
    urlcolor=black,
    citecolor=black,
    pdfpagemode=FullScreen,
}

\newtheorem{thm}{Theorem}[section]

\newtheorem{rmk}{Remark}[section]

\newtheorem{prop}{Proposition}[section]
\newtheorem{lm}{Lemma}[section]

\numberwithin{equation}{section}

\fancypagestyle{firstpage}{
\fancyhf{} 

\fancyfoot[L]{%
\footnotesize
\rule{2cm}{0.8pt}\\
This article is based upon work from the project PRIN2022 D53D23005580006 ``Elliptic and parabolic problems, heat kernel estimates and spectral theory''. The third author is member of the ``Gruppo Nazionale per l’Analisi Matematica, la Probabilità e le loro Applicazioni (GNAMPA)'' of the Istituto Nazionale di Alta Matematica (INdAM).\\[0.15cm]
$^\dagger$ Dipartimento di Matematica, Università degli Studi di Salerno,  Via Giovanni Paolo II, 132, 84084 Fisciano (SA), Italy.  \texttt{aelgrou@unisa.it}, \texttt{arhandi@unisa.it} \\
$^\ddagger$ Moulay Ismaïl University of Meknes, FST Errachidia, PAM Laboratory, BP 50, Boutalamine, Errachidia, Morocco. \texttt{omar.oukdach@gmail.com}
}
}

\title{\LARGE\textbf{A Hierarchical Robust Control Strategy for Stochastic Kuramoto--Sivashinsky--Korteweg--de Vries Equations}}

\author{Abdellatif Elgrou$^\dagger$,\,\, Omar Oukdach$^{\ddagger}$ \,\,and\,\, Abdelaziz Rhandi$^\dagger$}

\date{}

\begin{document}

\maketitle
\vspace{-3em}

\thispagestyle{firstpage}
\begin{center}
\end{center}

\begin{abstract}
We investigate the robust Stackelberg null controllability of a one-dimensional forward linear stochastic Kuramoto--Sivashinsky--Korteweg--de Vries (KS--KdV) equation. The control framework is formulated as a hierarchical Stackelberg game involving two leaders, one follower, and worst-case disturbances acting in both the drift and diffusion terms. The first leader acts to drive the system to rest, while the second leader is introduced to overcome analytical difficulties arising from the stochastic setting. The follower, by reducing the effect of the disturbances, addresses a tracking-type control problem aimed at keeping the system state and its first and second spatial derivatives close to prescribed target trajectories. First, the robust control problem is characterized by the existence of a saddle point. Then, the analysis is reduced to the null controllability of a strongly coupled forward--backward stochastic KS--KdV system. The problem is addressed by combining a duality technique with new Carleman estimates for forward and backward stochastic fourth-order parabolic equations.
\end{abstract}
\maketitle
\smallskip


\section{Introduction}
The Kuramoto--Sivashinsky--Korteweg--de Vries equation is a fourth-order nonlinear PDE of the form
\begin{equation}\label{model}
    y_t + k\, y_{xx} + y_{xxx} + \eta\, y_{xxxx} + y\,y_x = 0,
\end{equation}
which combines features of the Kuramoto--Sivashinsky (KS) and Korteweg--de Vries (KdV) equations to model diverse physical phenomena.  Throughout this paper, we refer to \eqref{model} as the KS–KdV system. It was first introduced by Benney \cite{benny66} and is therefore also known as the Benney equation. This equation has numerous applications in areas such as plasma physics and hydrodynamics; see, for instance, \cite{malfengkaw00133} and the references therein. The KS equation appears in flame front propagation, reaction--diffusion dynamics, and thin film evolution. It was introduced by Kuramoto \cite{kur78} and Kuramoto and Tsuzuki \cite{kur75tuz, kur76tuzu} to study phase turbulence, and later by Sivashinsky \cite{Ssivan4} in planar flame front stability analysis. The KdV equation, derived by Boussinesq \cite{bous77} and Korteweg--de Vries \cite{korVries895}, models the propagation of water waves in channels. 

In \eqref{model}, the function $y = y(t,x)$ represents the state variable. The fourth-order term $\eta\,y_{xxxx}$, with $\eta>0$, provides stabilizing dissipation, while the third-order term $y_{xxx}$ accounts for dispersion. In contrast, the second-order term $k\,y_{xx}$, with $k>0$, introduces anti-diffusion and may destabilize the dynamics. The nonlinear convection term $y\,y_x$ further enriches the system, giving rise to phenomena such as pattern formation and chaotic evolution. The interplay of destabilization, higher-order dissipation, dispersion, and nonlinearity endows the  KS--KdV system with a rich mathematical structure, making it a paradigmatic model in the study and control of nonlinear partial differential equations; see, e.g., \cite{KS23e,KS2,KS3}.

In many realistic situations, the system dynamics are influenced by random effects. This motivates the study of stochastic versions of the KS--KdV equation, where random perturbations model environmental noise or uncertainties in physical parameters. Such stochastic KS--KdV models arise, for instance, in dispersive-dissipative systems subject to random fluctuations. For further discussions on stochastic KS and KdV equations, we refer the reader to \cite{duanervin,Gaochenli15,Liangwu22,laucuma96,hevicto1} and the references therein.

Control problems with multiple objectives naturally lead to solution concepts based on strategic interactions. Classical notions in this context include Pareto optimality, Nash equilibria, and Stackelberg strategies; see \cite{StDf,Na51,Pa96,St34} and references therein. In particular, Stackelberg strategies describe hierarchical decision processes in which a leader makes a decision while anticipating the optimal response of a follower.

Another central aspect in control theory is robustness with respect to uncertainties and external disturbances. Robust control, initially developed for finite-dimensional systems in the late 1970s and early 1980s, seeks control strategies that maintain satisfactory performance despite modeling errors or external perturbations. For distributed parameter systems, such problems are often formulated as noncooperative games between the controller and the disturbance, leading naturally to a worst-case analysis where the disturbance opposes the control. The objective is then to determine a saddle point of the associated cost functional; see \cite{bewl12,bewl12se,bewl1} for further discussions on robust and optimal control of PDEs.

Motivated by these considerations, it is natural to investigate robust hierarchical control strategies for stochastic distributed systems. In this work, we study, for the first time, a robust Stackelberg controllability problem for a stochastic  KS--KdV equation. The proposed framework integrates a hierarchical decision mechanism with robustness against external disturbances.

Now, we introduce some necessary notation. Let $T>0$ and denote $Q := (0,T) \times (0,1)$. Let $\mathcal{O}$ and $\mathcal{D}$ be nonempty open subsets of $(0,1)$ representing the control regions, with $\mathcal{O} \cap \mathcal{D} = \emptyset$. For any subset $\mathcal{E}\subset(0,1)$, we denote by $\chi_{\mathcal{E}}$ its characteristic function. Finally, for a Hilbert space $\mathcal{H}$, we denote by $\|\cdot\|_{\mathcal{H}}$ and $\langle\cdot,\cdot\rangle_{\mathcal{H}}$ its norm and inner product, respectively. Let $(\Omega,\mathcal{F},\{\mathcal{F}_t\}_{t\in[0,T]},\mathbb{P})$ be a complete filtered probability space on which a one-dimensional standard Brownian motion $W(\cdot)$ is defined. The filtration $\{\mathcal{F}_t\}_{t\in[0,T]}$ is assumed to be the natural filtration generated by $W(\cdot)$ and augmented by all $\mathbb{P}$-null sets in $\mathcal{F}$. For a Banach space $\mathcal{X}$, we introduce the following spaces: $C([0,T];\mathcal{X})$ denotes the Banach space of all $\mathcal{X}$-valued continuous functions defined on $[0,T]$; $L^2_{\mathcal{F}_t}(\Omega;\mathcal{X})$ denotes the space of all $\mathcal{X}$-valued, $\mathcal{F}_t$-measurable random variables $X$ such that $\mathbb{E}\big(\|X\|_{\mathcal{X}}^2\big)<\infty$; $L^2_{\mathcal{F}}(0,T;\mathcal{X})$ denotes the space of all $\mathcal{X}$-valued, $\{\mathcal{F}_t\}_{t\in[0,T]}$-adapted processes $X(\cdot)$ satisfying $\mathbb{E}\int_0^T \|X(t)\|_{\mathcal{X}}^2\,dt < \infty$; $L^\infty_{\mathcal{F}}(0,T;\mathcal{X})$ denotes the space of all essentially bounded and adapted $\mathcal{X}$-valued processes, endowed with the norm $\|\cdot\|_{\infty}$; and $L^2_{\mathcal{F}}(\Omega;C([0,T];\mathcal{X}))$ denotes the space of all $\mathcal{X}$-valued, adapted continuous processes $X(\cdot)$ such that $\mathbb{E}\!\left[\max_{t\in[0,T]} \|X(t)\|_{\mathcal{X}}^2\right] < \infty$.

Consider the following linear stochastic KS--KdV equation:
\begin{equation}\label{eqq1.1}
\begin{cases}
\begin{array}{lll}
dy +(k\,y_{xx}+  y_{xxx}+\eta\,y_{xxxx})\,dt
= \left[ay+ f\chi_{\mathcal{O}}
+ v\chi_{\mathcal{D}}+\psi_1\right] dt+ \left[by+g+\psi_2\right]\,dW(t),
& (t,x)\in Q,\\[2mm]
y(t,0)=y(t,1)=0, & t\in(0,T),\\
y_x(t,0)=y_x(t,1)=0, & t\in(0,T),\\
y(0,x)=y_0(x), & x\in(0,1),
\end{array}
\end{cases}
\end{equation}
where $y_0 \in L^2_{\mathcal{F}_0}(\Omega;L^2(0,1))$ denotes the initial state, and 
$\psi_1,\psi_2 \in L^2_{\mathcal{F}}(0,T;L^2(0,1))$ represent unknown disturbances acting, respectively, on the drift and diffusion terms, and $k, \eta \in \mathbb{R}$ are positive constants. The pair $(f,g) \in L^2_{\mathcal{F}}(0,T;L^2(\mathcal{O})) \times L^2_{\mathcal{F}}(0,T;L^2(0,1))$ corresponds to the leader controls, while $v \in L^2_{\mathcal{F}}(0,T;L^2(\mathcal{D}))$ corresponds to the follower control. The zero-order coefficients $a,b \in L^\infty_{\mathcal{F}}(0,T;L^\infty(0,1))$. The stochastic perturbation driven by the Wiener process $W(t)$ accounts for random fluctuations or uncertainties, while the control inputs $f$, $g$, and $v$ act on the subdomains $\mathcal{O}$ and $\mathcal{D}$ to influence the system behavior.

The system \eqref{eqq1.1} involves three controls: the pair $(f,g)$, referred to as the \emph{leaders}, and $v$, called the \emph{follower}. The control objectives are hierarchical in nature. The primary objective is of controllability type, whereas the secondary objective corresponds to a robust optimal tracking problem. More precisely, the first leader 
$f \in L^2_{\mathcal{F}}(0,T;L^2(\mathcal{O}))$ acts on the subregion $\mathcal{O}$ and aims to drive the system to rest at time $T$. The second leader 
$g \in L^2_{\mathcal{F}}(0,T;L^2(0,1))$ acts on the whole domain and is introduced to overcome certain technical difficulties caused by the stochastic perturbation. The follower 
$v \in L^2_{\mathcal{F}}(0,T;L^2(\mathcal{D}))$ acts on the subregion $\mathcal{D}$ and aims to keep the state $y$, together with its first and second spatial derivatives $y_x$ and $y_{xx}$, close to prescribed targets, while compensating for the disturbance generated by the source terms $\psi_1$ and $\psi_2$. The disturbance associated with $\psi_2$ in the diffusion term introduces additional difficulties, requiring a new Carleman estimate (see Remark \ref{rmk1.1.} for further details).

Similarly to \cite{Gaochenli15}, one can show that the system \eqref{eqq1.1} is well-posed. More precisely, for any $y_0 \in L^2_{\mathcal{F}_0}(\Omega;L^2(0,1))$, $(f,g) \in L^2_{\mathcal{F}}(0,T;L^2(\mathcal{O})) \times L^2_{\mathcal{F}}(0,T;L^2(0,1))$, $v \in L^2_{\mathcal{F}}(0,T;L^2(\mathcal{D}))$, and  $\psi_1,\psi_2 \in L^2_{\mathcal{F}}(0,T;L^2(0,1))$, there exists a unique weak solution 
\[
y \in L^2_{\mathcal{F}}(\Omega;C([0,T];L^2(0,1))) \bigcap L^2_{\mathcal{F}}(0,T;H^2_0(0,1)).
\]
Moreover, there exists a constant $C>0$, independent of the data, such that
\begin{align*}
&\|y\|_{L^2_{\mathcal{F}}(\Omega;C([0,T];L^2(0,1)))} 
+ \|y\|_{L^2_{\mathcal{F}}(0,T;H^2_0(0,1))} \\
&\quad \le C \Big(
\|y_0\|_{L^2_{\mathcal{F}_0}(\Omega;L^2(0,1))} 
+ \|f\|_{L^2_{\mathcal{F}}(0,T;L^2(\mathcal{O}))} 
+ \|g\|_{L^2_{\mathcal{F}}(0,T;L^2(0,1))} 
\\
&\hspace{1.4cm}+ \|v\|_{L^2_{\mathcal{F}}(0,T;L^2(\mathcal{D}))} 
+ \|\psi_1\|_{L^2_{\mathcal{F}}(0,T;L^2(0,1))}+ \|\psi_2\|_{L^2_{\mathcal{F}}(0,T;L^2(0,1))} 
\Big).
\end{align*}

We now formulate the robust Stackelberg null controllability problem for \eqref{eqq1.1}. Let 
\(y_d^i \in L^2_{\mathcal{F}}(0,T;L^2(\mathcal{O}_d^i))\), \(i=0,1,2\), denote target functions defined respectively on the observation regions \(\mathcal{O}_d^i \subset (0,1)\), which are assumed to be nonempty and open sets. For the follower \(v\) and the disturbances \(\psi_1, \psi_2\), we introduce the secondary (robust) cost functional:
\begin{align}\label{robustcost}
\begin{aligned}
\mathcal{J}_r(f,g;\psi_1,\psi_2,v) 
&:= \frac{1}{2}\,\mathbb{E}\iint_Q \Big(\big|y-y_d^0\big|^2\chi_{\mathcal{O}_d^0} + \big|y_x-y_d^1\big|^2\chi_{\mathcal{O}_d^1}+ \big|y_{xx}-y_d^2\big|^2\chi_{\mathcal{O}_d^2}\Big) \,dx\, dt \\
&\hspace{0.6cm}+ \frac{\beta}{2}\,\mathbb{E}\int_0^T\int_{\mathcal{D}} |v|^2 \,dx\, dt
     - \frac{\delta_1}{2}\,\mathbb{E}\iint_Q |\psi_1|^2 \,dx\, dt
   - \frac{\delta_2}{2}\,\mathbb{E}\iint_Q |\psi_2|^2 \,dx\, dt,
   \end{aligned}
\end{align}
where $\beta,\delta_1,\delta_2 > 0$ are constants, and $y = y(y_0,f,g,\psi_1,\psi_2,v)$ denotes the solution of \eqref{eqq1.1} corresponding to the given controls and disturbances.

The primary cost functional associated with the leader control pair $(f,g)$ is defined by
\begin{equation}\label{leaderscost}
\mathcal{J}(f,g) := \frac{1}{2}\,\mathbb{E}\iint_Q \big(|f|^2\chi_{\mathcal{O}} +|g|^2\big) \,dx\, dt.
\end{equation}

The robust Stackelberg strategy is formulated in a hierarchical manner: for a fixed pair of leaders $(f,g) \in L^2_{\mathcal{F}}(0,T;L^2(\mathcal{O})) \times L^2_{\mathcal{F}}(0,T;L^2(0,1))$, the first step is to determine a saddle point $(\psi_1^\star,\psi_2^\star,v^\star) \in \mathcal{H}^{\psi} \times \mathcal{H}^v$ satisfying
\begin{equation}\label{miniprobofjtil}
\mathcal{J}_r(f,g;\psi_1^\star,\psi_2^\star,v^\star)
= \max_{(\psi_1,\psi_2) \in \mathcal{H}^{\psi}} \min_{v \in \mathcal{H}^v} \mathcal{J}_r(f,g;\psi_1,\psi_2,v)
= \min_{v \in \mathcal{H}^v} \max_{(\psi_1,\psi_2) \in \mathcal{H}^{\psi}} \mathcal{J}_r(f,g;\psi_1,\psi_2,v),
\end{equation}
where the admissible sets for the disturbances and the follower are given by
\[
\mathcal{H}^{\psi_1}=\mathcal{H}^{\psi_2} := L^2_{\mathcal{F}}(0,T;L^2(0,1)), 
\qquad \mathcal{H}^{\psi}:=\mathcal{H}^{\psi_1}\times\mathcal{H}^{\psi_2},\qquad
\mathcal{H}^v := L^2_{\mathcal{F}}(0,T;L^2(\mathcal{D})).
\]

Once the saddle point $(\psi_1^\star,\psi_2^\star,v^\star)$ has been determined for each fixed pair of leaders $(f,g)$, the next step is to select the optimal leaders $(\widehat{f},\widehat{g}) \in L^2_{\mathcal{F}}(0,T;L^2(\mathcal{O})) \times L^2_{\mathcal{F}}(0,T;L^2(0,1))$ that minimize the primary cost functional $\mathcal{J}$ while accounting for the follower's optimal response. That is,
\[ \mathcal{J}(\widehat{f},\widehat{g}) = \min_{(f,g)\in \mathcal{H}^{f,g}} \mathcal{J}(f,g), \]
where
\[
\mathcal{H}^{f,g} := L^2_{\mathcal{F}}(0,T;L^2(\mathcal{O})) \times L^2_{\mathcal{F}}(0,T;L^2(0,1)),
\]
and the corresponding solution $y = y(y_0,\widehat{f},\widehat{g},\psi_1^\star,\psi_2^\star,v^\star)$ of \eqref{eqq1.1} satisfies the null controllability condition
\begin{equation}\label{ncontrol}
y(T,\cdot) = 0 \quad \text{in } (0,1), \quad \text{a.s.}
\end{equation}

The objective of the robust control problem \eqref{miniprobofjtil} is to determine a follower control $v^\star$ that steers the system state $y$, together with its spatial derivatives $y_x$ and $y_{xx}$, close to the prescribed target trajectories $y^0_d$, $y^1_d$, and $y^2_d$, respectively, while accounting for unknown disturbances $\psi_1$ and $\psi_2$. The disturbances act adversarially, attempting to maximize the deviation from the desired targets. Consequently, the follower minimizes the cost functional with respect to $v$, whereas the disturbances maximize it with respect to $(\psi_1,\psi_2)$. Mathematically, this gives rise to a min--max optimization problem, whose solution is characterized by the saddle point $(\psi_1^\star,\psi_2^\star,v^\star)$ of the cost functional $\mathcal{J}_r$. The triple $(\psi_1^\star,\psi_2^\star,v^\star)$ represents the equilibrium between the optimal follower control and the worst-case disturbances within the robust Stackelberg framework.

The analysis of our hierarchical robust control problem is reduced to the null controllability of a strongly coupled forward--backward stochastic KS--KdV system. This problem is addressed by means of suitable Carleman estimates for fourth-order stochastic parabolic equations, combined with duality arguments. To solve our problem, we assume the following geometric condition:

\begin{equation}\label{Assump10}
\mathcal{O} \cap \mathcal{O}_d^0 \neq \emptyset,\quad\textnormal{and}\quad\mathcal{O} \cap \mathcal{O}_d^0\nsubseteq(\mathcal{O}_d^1\cup\mathcal{O}_d^2).
\end{equation}

We now state the main result of this paper.
\begin{thm}[Robust Stackelberg strategy]\label{th4.1SN}
Assume that condition \eqref{Assump10} holds, and that $\beta,\delta_1,\delta_2>0$ are sufficiently large. Then, there exists a positive weight function $\rho=\rho(t)$ blowing up as $t \to T$ such that, for any initial condition $y_0\in L^2_{\mathcal{F}_0}(\Omega;L^2(0,1))$ and target functions $y_d^i\in L^2_{\mathcal{F}}(0,T;L^2(\mathcal{O}_d^i))$ ($i=0,1,2$) satisfying
\begin{equation}\label{inqAss11SN}
\mathbb{E}\iint_Q \rho^2 \Big(\big|y_d^0\big|^2\chi_{\mathcal{O}_d^0}+\big|y_d^1\big|^2\chi_{\mathcal{O}_d^1}+\big|y_d^2\big|^2\chi_{\mathcal{O}_d^2}\Big) \,dx\,dt < \infty,
\end{equation}
there exist controls $(\widehat{f},\widehat{g})\in \mathcal{H}^{f,g}$ minimizing $\mathcal{J}$ and a unique saddle point $(\psi_1^\star,\psi_2^\star,v^\star) \in \mathcal{H}^{\psi}\times\mathcal{H}^{v}$  for the functional defined in \eqref{robustcost}, such that the corresponding solution $\widehat{y}= \widehat{y}(y_0,\widehat{f},\widehat{g},\psi_1^\star,\psi_2^\star,v^\star)$ of \eqref{eqq1.1} satisfies
\[
\widehat{y}(T,\cdot)=0 \quad \textnormal{in } (0,1), \quad \textnormal{a.s.}
\]
Moreover, the controls $(\widehat{f},\widehat{g})$ satisfy the estimate
\begin{align*}
&\|\widehat{f}\|^2_{L^2_{\mathcal{F}}(0,T;L^2(\mathcal{O}))} 
+ \|\widehat{g}\|^2_{L^2_{\mathcal{F}}(0,T;L^2(0,1))} \\
&\;\le C_T \bigg[ \mathbb{E}\|y_0\|^2_{L^2(0,1)} + \mathbb{E}\iint_Q \rho^2 \Big(\big|y_d^0\big|^2\chi_{\mathcal{O}_d^0}+\big|y_d^1\big|^2\chi_{\mathcal{O}_d^1}+\big|y_d^2\big|^2\chi_{\mathcal{O}_d^2}\Big) \,dx\,dt  \bigg],
\end{align*}
where $C_T>0$ is a constant depending on $\mathcal{O}$, $\mathcal{D}$, $\mathcal{O}_d^0$, $\mathcal{O}_d^1$, $\mathcal{O}_d^2$, $T$, $\|a\|_\infty$, $\|b\|_\infty$.
\end{thm}

As a consequence, we also obtain the following Stackelberg strategy in the absence of disturbances, that is, when $\psi_1=\psi_2=0$. In this case, the follower cost functional reduces to
\begin{align}\label{follfonly}
\widetilde{\mathcal{J}}(f,g;v) = \frac{1}{2}\,\mathbb{E}\iint_Q \Big(\big|y-y_d^0\big|^2\chi_{\mathcal{O}_d^0} + \big|y_x-y_d^1\big|^2\chi_{\mathcal{O}_d^1}+ \big|y_{xx}-y_d^2\big|^2\chi_{\mathcal{O}_d^2}\Big) \,dx\, dt + \frac{\beta}{2} \mathbb{E}\int_0^T\int_{\mathcal{D}} |v|^2 \,dx\, dt.
\end{align}

\begin{prop}[Stackelberg strategy]\label{th4.2SN}
Assume that  \eqref{Assump10} holds, and that $\beta>0$ is sufficiently large. Then, there exists a positive weight function $\rho=\rho(t)$ blowing up as $t \to T$ such that, for any initial condition $y_0 \in L^2_{\mathcal{F}_0}(\Omega;L^2(0,1))$ and any target functions $y_d^i\in L^2_{\mathcal{F}}(0,T;L^2(\mathcal{O}_d^i))$ ($i=0,1,2$) satisfying
\begin{equation*}
\mathbb{E}\iint_Q \rho^2 \Big(\big|y_d^0\big|^2\chi_{\mathcal{O}_d^0}+\big|y_d^1\big|^2\chi_{\mathcal{O}_d^1}+\big|y_d^2\big|^2\chi_{\mathcal{O}_d^2}\Big) \,dx\,dt < \infty,
\end{equation*}
there exist controls $(\widehat{f},\widehat{g}) \in \mathcal{H}^{f,g}$ minimizing $\mathcal{J}$ and a unique follower control $v^\star \in \mathcal{H}^v$ minimizing \eqref{follfonly} such that the corresponding solution $\widehat{y} = \widehat{y}(y_0, \widehat{f}, \widehat{g}, v^\star)$ of \eqref{eqq1.1} with $\psi_1=\psi_2=0$ satisfies
\[
\widehat{y}(T,\cdot) = 0 \quad \textnormal{in } (0,1), \quad \textnormal{a.s.}
\]
Moreover, the estimate for the leaders' controls $(\widehat{f},\widehat{g})$ is similar to that in Theorem \ref{th4.1SN}.
\end{prop}

Hierarchical control problems constitute a central topic in optimal control theory. 
They were first introduced by J.-L. Lions in \cite{LiPa} and have since generated a vast literature. 
Given this extensive body of work, particularly in the deterministic setting, we only mention a few representative contributions. 
Stackelberg control for parabolic equations in unbounded domains was studied in \cite{unbounded}, and the interaction between Nash and Stackelberg strategies has been widely investigated. 
For instance, Nash equilibria in multi-objective control of linear partial differential equations were analyzed in \cite{GRP02}. 
Stackelberg--Nash strategies for parabolic equations have been extensively studied, including exact controllability of linear equations in \cite{ArFeGu17}, controllability for linear and semilinear equations in \cite{ArCaSa15}, and further developments for parabolic systems in \cite{Calsavara,HSP18}. 
Additional contributions include Stackelberg--Nash null controllability of the heat equation with dynamic boundary conditions in \cite{BoMaOuNash,BoMaOuNash2}.

For deterministic KS and KdV equations, we refer to \cite{KSStack2,StackKS} for hierarchical control problems, where Stackelberg--Nash controllability results are obtained with distributed and boundary controls, and to \cite{albarasan23} for the KdV equation. 
Regarding robust Stackelberg controllability, we mention \cite{refee11} for linear and semilinear heat equations, \cite{refee12} for the Navier--Stokes equations, \cite{refee13} for heat equations with boundary controls, and \cite{refee14} for the KS equation. 

Hierarchical and multi-objective control problems under Nash or Stackelberg--Nash strategies in the stochastic setting have only recently attracted attention; see, for instance, \cite{Nashstoch3,Nashstoch1,Nashstoch2} for stochastic parabolic equations, \cite{yuzhang23} for stochastic degenerate parabolic equations, and \cite{Nashstoch4} for stochastic evolution equations.

To the best of our knowledge, this is the first work addressing a robust Stackelberg controllability problem for stochastic partial differential equations. The present study extends deterministic results to the stochastic setting, tackling new challenges and providing a framework for hierarchical multi-objective robust control under stochastic perturbations, with particular emphasis on the stochastic KS--KdV equation.

\begin{rmk}\label{rmk1.1.}
In the presence of the disturbance $\psi_2 \neq 0$, the Carleman estimate in \cite[Theorem 3.1]{Gaochenli15} cannot be directly applied to prove the robust Stackelberg controllability of \eqref{eqq1.1} in Theorem \ref{th4.1SN}. To address this, in Section \ref{sec3sen} we derive an improved Carleman estimate (see Theorem \ref{themmain5.1}) that requires the diffusion term only in $L^2(0,1)$ instead of $H^2(0,1)$. This estimate is crucial in Section \ref{sec3} to obtain the  Carleman estimate in Lemma \ref{thmm5.1} for the coupled system \eqref{ADJSO1}, since the diffusion term in the forward equation involves $P$, and a direct application of \cite[Theorem 3.1]{Gaochenli15} would require unavailable estimates for $P_x$ and $P_{xx}$. Moreover, it allows the spatial regularity of the coefficient $b$ to be reduced from $W^{2,\infty}(0,1)$ to $L^\infty(0,1)$. Furthermore, in order to handle both the KdV mean term $y_{xxx}$ in \eqref{eqq1.1} and the second derivative term $y_{xx}$ in \eqref{robustcost}, we employ in Section \ref{sec3} the new derived Carleman estimates for both the forward equation (Theorem \ref{themmain5.1}) and the backward equation (Theorem \ref{thmm3.2p}) for stochastic fourth-order parabolic equations.
\end{rmk}
\begin{rmk}
In \eqref{eqq1.1}, although the follower control \(v\) acts on the drift term, the results of this paper remain valid when the control \(v\) acts instead on the diffusion term.
\end{rmk}

The results obtained in this paper naturally lead to several challenging questions for future research:
\begin{itemize}
\item The assumption \(\mathcal{O} \cap \mathcal{O}_d^0 \nsubseteq (\mathcal{O}_d^1\cup\mathcal{O}_d^2)\) in \eqref{Assump10} ensures a partial separation between the observation regions \(\mathcal{O}_d^0\), \(\mathcal{O}_d^1\) and \(\mathcal{O}_d^2\), corresponding to the state and its first and second derivatives, respectively. This condition is crucial for deriving the main Carleman estimate \eqref{Carlem5.9} and, in particular, excludes the cases where \(\mathcal{O}_d^0 = \mathcal{O}_d^1\) and/or \(\mathcal{O}_d^0 = \mathcal{O}_d^2\), which remain to be addressed.
\item Studying the robust Stackelberg controllability of \eqref{eqq1.1} when the follower control \(v\) and/or the leader control \(f\) acts locally on the boundary leads to a classical controllability problem for KS--KdV systems coupled through the boundary. This setting requires Carleman estimates of a different nature from those developed in Section~\ref{sec3sen}, namely, estimates adapted to fourth-order parabolic equations with nonhomogeneous boundary conditions.
\item The controllability result in Theorem \ref{th4.1SN} is obtained using two leader controls \(f\) and \(g\) acting on the drift and diffusion terms, respectively. It would be of interest to investigate whether the Lebeau--Robbiano strategy \cite{lu2011some} allows one to achieve the result with a single leader control. Moreover, the present work is restricted to the linear stochastic KS--KdV equation. Extending the analysis to the semilinear case with the nonlinear transport term \(y\,y_x\) remains open, mainly due to its non-globally Lipschitz nature and the lack of compactness (\cite[Remark 2.5]{Tangzhang9}) in the functional framework. Classical deterministic techniques do not directly extend to the stochastic setting, and such an extension would likely require truncation methods, fixed-point arguments, and suitable a priori estimates \cite{hevicto12}.
\end{itemize}

The rest of this paper is organized as follows. In Section \ref{sec2}, we study the well-posedness of the robust follower control problem and characterize its solution. Section \ref{sec3sen} presents improved Carleman estimates for one-dimensional forward and backward stochastic fourth-order parabolic equations. In Section \ref{sec3}, we derive the main Carleman estimate for a strongly coupled stochastic KS–KdV system. Finally, Section \ref{sec4} is devoted to the proof of the main result.

\section{Robust Follower Control: Well-Posedness and Characterization}\label{sec2}

In this section, we analyze the robust follower control problem for a fixed pair of leader controls $(f,g) \in \mathcal{H}^{f,g}$. Our objective is to establish the existence, uniqueness, and explicit characterization of the optimal follower control $v^\star$ together with the corresponding adversarial disturbances $\psi_1^\star$ and $\psi_2^\star$. Specifically, we seek the saddle point $(\psi_1^\star,\psi_2^\star,v^\star)$ that simultaneously minimizes the deviation of the system from a prescribed target while accounting for the worst-case effect of the disturbances. 

The proof of existence and uniqueness of a solution $(\psi_1^\star, \psi_2^\star, v^\star)$ to the robust optimal control problem \eqref{miniprobofjtil} relies on the following classical result (see \cite{eketm99}).
\begin{lm}\label{lem:saddle_point}
Let $J$ be a functional defined on $X \times Y$, where $X$ and $Y$ are non-empty, closed, unbounded convex sets. Assume that $J$ satisfies the following conditions:
\begin{enumerate}[(a)]
    \item For all $\psi \in X$, the mapping $v \mapsto J(\psi,v)$ is convex and lower semicontinuous.
    \item For all $v \in Y$, the mapping $\psi \mapsto J(\psi,v)$ is concave and upper semicontinuous.
    \item There exists $\psi_0 \in X$ such that 
    \[
    \lim_{\|v\|_Y \to +\infty} J(\psi_0, v) = +\infty.
    \]
    \item There exists $v_0 \in Y$ such that 
    \[
    \lim_{\|\psi\|_X \to +\infty} J(\psi, v_0) = -\infty.
    \]
\end{enumerate}
Then $J$ admits at least one saddle point $(\psi^\star, v^\star)$ satisfying
\[
J(\psi^\star, v^\star) = \min_{v \in Y} \max_{\psi \in X} J(\psi,v) = \max_{\psi \in X} \min_{v \in Y} J(\psi,v), \quad \forall (\psi,v)\in X\times Y.
\]
Moreover, if $v \mapsto J(\psi,v)$ is strictly convex for all $\psi \in X$, and $\psi \mapsto J(\psi,v)$ is strictly concave for all $v \in Y$, then the saddle point $(\psi^\star, v^\star)$ is unique.
\end{lm}

Applying Lemma \ref{lem:saddle_point} with $X = \mathcal{H}^{\psi}$ and $Y = \mathcal{H}^v$, it is straightforward to verify that the cost functional $\mathcal{J}_r$ in \eqref{miniprobofjtil} satisfies all the required assumptions. Consequently, we obtain the following result on the existence and uniqueness of the saddle point.

\begin{prop}[Existence and Uniqueness of the Follower–Disturbance Saddle Point]\label{propp4.1sapoi}
Let $(f,g)\in \mathcal{H}^{f,g}$ be fixed and let $y_0 \in L^2_{\mathcal{F}_0}(\Omega;L^2(0,1))$. For sufficiently large $\delta_1,\delta_2>0$, there exists a unique saddle point $(\psi_1^\star,\psi_2^\star, v^\star) \in \mathcal{H}^{\psi} \times \mathcal{H}^v$ and a corresponding solution $y(y_0,f,g,\psi_1^\star,\psi_2^\star, v^\star)$ of \eqref{eqq1.1} such that
\[
\mathcal{J}_r(f,g;\psi_1^\star,\psi_2^\star, v^\star) = \max_{(\psi_1,\psi_2) \in \mathcal{H}^{\psi}} \min_{v \in \mathcal{H}^v} \mathcal{J}_r(f,g;\psi_1,\psi_2,v) = \min_{v \in \mathcal{H}^v} \max_{(\psi_1,\psi_2) \in \mathcal{H}^{\psi}} \mathcal{J}_r(f,g;\psi_1,\psi_2,v).
\]
\end{prop}

Let us now characterize the saddle point $(\psi_1^{\star},\psi_2^{\star},v^{\star})$ obtained in Proposition \ref{propp4.1sapoi}.  
Since the functional $\mathcal{J}_r$ is Fréchet differentiable on $\mathcal{H}^{\psi} \times \mathcal{H}^v$, the fact that $(\psi_1^{\star},\psi_2^{\star},v^{\star})$ is a saddle point implies the first-order optimality conditions: 
\begin{align}
\label{gatder1_psi1}
\frac{\partial \mathcal{J}_r}{\partial \psi_i}(f,g;\psi_1^{\star},\psi_2^{\star},v^{\star})(u_i) &= 0, 
\quad \forall u_i \in \mathcal{H}^{\psi_i},\quad i=1,2,\\
\label{gatder_v}
\frac{\partial \mathcal{J}_r}{\partial v}(f,g;\psi_1^{\star},\psi_2^{\star},v^{\star})(w) &= 0, 
\quad \forall w \in \mathcal{H}^v.
\end{align}
By computing \eqref{gatder1_psi1}, one readily obtains
\begin{align}
\label{equau11_psi1}
\sum_{j=0}^2\mathbb{E}\iint_Q  (\partial_x^j y-y_d^{j})\partial_x^j\xi^{u_i}\chi_{\mathcal{O}_d^{j}} \,dx\,dt 
- \delta_i\,\mathbb{E}\iint_Q u_i\,\psi_i^{\star} \,dx\, dt &= 0,
\quad \forall u_i\in \mathcal{H}^{\psi_i},\quad i=1,2,
\end{align}
where $y$ denotes the solution of \eqref{eqq1.1} associated with $f$, $g$, $\psi_1^{\star}$, $\psi_2^{\star}$, and $v^{\star}$, and $\xi^{u_1}$ and $\xi^{u_2}$ are, respectively, the solutions of the forward stochastic KS--KdV systems:
\begin{equation}\label{1.10psi1}
\begin{cases}
d\xi^{u_1} +(k\,\xi^{u_1}_{xx}+  \xi^{u_1}_{xxx}+\eta\,\xi^{u_1}_{xxxx})\,dt
= \big[ a\xi^{u_1}  + u_1 \big] dt +b\xi^{u_1} dW(t), & (t,x)\in Q,\\
\xi^{u_1}(t,0)=\xi^{u_1}(t,1)=0, & t\in(0,T),\\
\xi^{u_1}_x(t,0)=\xi^{u_1}_x(t,1)=0, & t\in(0,T),\\
\xi^{u_1}(0,x)=0, & x\in(0,1),
\end{cases}
\end{equation}
and 
\begin{equation}\label{1.10psi2}
\begin{cases}
d\xi^{u_2} +(k\,\xi^{u_2}_{xx}+  \xi^{u_2}_{xxx}+\eta\,\xi^{u_2}_{xxxx})\,dt
= a\xi^{u_2}  dt + \big[ b\xi^{u_2}  + u_2 \big] dW(t), & (t,x)\in Q,\\
\xi^{u_2}(t,0)=\xi^{u_2}(t,1)=0, & t\in(0,T),\\
\xi^{u_2}_x(t,0)=\xi^{u_2}_x(t,1)=0, & t\in(0,T),\\
\xi^{u_2}(0,x)=0, & x\in(0,1).
\end{cases}
\end{equation}
Similarly, for the follower control $v^\star$, we have 
\begin{align}\label{equau11_v}
\sum_{j=0}^2\mathbb{E}\iint_Q  (\partial_x^j y-y_d^{j})\partial_x^j\xi^{w}\chi_{\mathcal{O}_d^{j}} \,dx\,dt  
+ \beta\,\mathbb{E}\iint_Q w\,v^{\star} \chi_{\mathcal{D}}\,dx\, dt = 0,
\quad \forall w\in \mathcal{H}^v,
\end{align}
where $\xi^w$ solves the system
\begin{equation}\label{1.10_v}
\begin{cases}
d\xi^w +(k\,\xi^w_{xx}+  \xi^w_{xxx}+\eta\,\xi^w_{xxxx})\,dt
= \big[ a\xi^w  + w\,\chi_{\mathcal{D}} \big] dt + b\xi^w  dW(t), & (t,x)\in Q,\\
\xi^w(t,0)=\xi^w(t,1)=0, & t\in(0,T),\\ \xi^w_x(t,0)=\xi^w_x(t,1)=0, & t\in(0,T),\\
\xi^w(0,x)=0, & x\in(0,1).
\end{cases}
\end{equation}
Next, we introduce the following adjoint backward stochastic KS--KdV system
\begin{equation}\label{backadj_two}
\begin{cases}
dz -(k z_{xx}-  z_{xxx}+\eta\,z_{xxxx})\,dt
= \bigg[ -az - bZ   +\displaystyle\sum_{i=0}^2 (-1)^{i+1}\partial_x^i(\partial_x^i y -y^{i}_d)\,\chi_{\mathcal{O}_d^{i}} \bigg] dt \\
\hspace{5.4cm}+ Z\,dW(t), & (t,x)\in Q,\\
z(t,0)=z(t,1)=0, & t\in(0,T),\\ z_x(t,0)=z_x(t,1)=0, & t\in(0,T),\\
z(T,x)=0, & x\in(0,1).
\end{cases}
\end{equation}
Applying Itô's formula to the products $\xi^{u_1} z$, $\xi^{u_2} z$, and $\xi^w z$, integrating over $Q$, and taking expectations, we obtain the variational identities:
\begin{align}
\label{eqqgatde1}
&\mathbb{E}\iint_Q z\, u_1 \,dx\, dt 
- \sum_{j=0}^2\mathbb{E}\iint_Q  (\partial_x^j y-y_d^{j})\partial_x^j\xi^{u_1}\chi_{\mathcal{O}_d^{j}} \,dx\,dt  = 0, 
\quad \forall u_1 \in \mathcal{H}^{\psi_1},\\[1mm]
\label{eqqgatde2}
&\mathbb{E}\iint_Q Z\, u_2 \,dx\, dt 
- \sum_{j=0}^2\mathbb{E}\iint_Q  (\partial_x^j y-y_d^{j})\partial_x^j\xi^{u_2}\chi_{\mathcal{O}_d^{j}} \,dx\,dt  = 0, 
\quad \forall u_2 \in \mathcal{H}^{\psi_2},\\[1mm]
\label{eqqgatde3}
&\mathbb{E}\iint_Q z\, w \chi_{\mathcal{D}} \,dx\, dt 
- \sum_{j=0}^2\mathbb{E}\iint_Q  (\partial_x^j y-y_d^{j})\partial_x^j\xi^{w}\chi_{\mathcal{O}_d^{j}} \,dx\,dt  = 0, 
\quad \forall w \in \mathcal{H}^v.
\end{align}
Combining \eqref{eqqgatde1}, \eqref{eqqgatde2}, and \eqref{eqqgatde3} with 
\eqref{equau11_psi1} (for $i=1,2$) and \eqref{equau11_v}, 
we obtain the following explicit characterization of the saddle point $(\psi_1^{\star},\psi_2^{\star},v^{\star})$:
\begin{equation*}
\psi_1^\star = \frac{1}{\delta_1}\, z, \qquad
\psi_2^\star = \frac{1}{\delta_2}\, Z, \qquad
v^\star = -\frac{1}{\beta}\, z \, \chi_{\mathcal{D}},
\end{equation*}
where $(z,Z)$ is the solution of the system \eqref{backadj_two}. 

Therefore, the robust Stackelberg optimal control problem for system \eqref{eqq1.1} is reduced to establishing the null controllability of the solutions to the following strongly coupled forward–backward stochastic KS--KdV system, which will be referred to as the \emph{optimality system}
\begin{equation}\label{eqq4.7}
\begin{cases}
\begin{array}{lll}
dy +(k\,y_{xx}+  y_{xxx}+\eta\,y_{xxxx})\,dt
= \left[ay+ f\chi_{\mathcal{O}}
-\frac{1}{\beta} z\chi_{\mathcal{D}}+\frac{1}{\delta_1}\, z\right] dt+ \big[by+g+\frac{1}{\delta_2}\, Z\big]\,dW(t),
& (t,x)\in Q,\\[2mm]
dz -(k\,z_{xx}-  z_{xxx}+\eta\,z_{xxxx})\,dt
= \bigg[-az-bZ+\displaystyle\sum_{i=0}^2 (-1)^{i+1}\partial_x^i(\partial_x^i y -y^{i}_d)\,\chi_{\mathcal{O}_d^{i}}\bigg] dt
\\
\hspace{5.4cm}+ Z\,dW(t),
& (t,x)\in Q,\\[2mm]
y(t,0)=y(t,1)=0, & t\in(0,T),\\
y_x(t,0)=y_x(t,1)=0, & t\in(0,T),\\
z(t,0)=z(t,1)=0, & t\in(0,T),\\
z_x(t,0)=z_x(t,1)=0, & t\in(0,T),\\
y(0,x)=y_0(x), & x\in(0,1),\\
z(T,x)=0, & x\in(0,1).
\end{array}
\end{cases}
\end{equation}

By classical duality arguments, the null controllability of system \eqref{eqq4.7} can be established by proving an appropriate observability inequality for the following associated coupled backward–forward stochastic KS--KdV system
\begin{equation}\label{ADJSO1}
\begin{cases}
\begin{array}{lll}
dp -(k\,p_{xx}-  p_{xxx}+\eta\,p_{xxxx})\,dt
= \big[-ap-bP+q\chi_{\mathcal{O}_d^0}-q_{xx}\chi_{\mathcal{O}_d^1}+q_{xxxx}\chi_{\mathcal{O}_d^2}\big] dt
+ P\,dW(t),
& (t,x)\in Q,\\[2mm]
dq +(k\,q_{xx}+  q_{xxx}+\eta\,q_{xxxx})\,dt
= \left[aq+\frac{1}{\beta} p\chi_{\mathcal{D}}-\frac{1}{\delta_1}\,p\right] dt+ \big[bq-\frac{1}{\delta_2}\,P\big]\,dW(t),
& (t,x)\in Q,\\[2mm]
p(t,0)=p(t,1)=0, & t\in(0,T),\\
p_x(t,0)=p_x(t,1)=0, & t\in(0,T),\\
q(t,0)=q(t,1)=0, & t\in(0,T),\\
q_x(t,0)=q_x(t,1)=0, & t\in(0,T),\\
p(T,x)=p_T(x), & x\in(0,1),\\
q(0,x)=0, & x\in(0,1).
\end{array}
\end{cases}
\end{equation}
\section{Carleman Estimates for Forward and Backward Stochastic Fourth-Order Parabolic Equations}\label{sec3sen}
In this section, we employ a duality approach to establish new Carleman estimates for one-dimensional forward and backward stochastic fourth-order parabolic equations. Here, the drift term is allowed to belong to the negative Sobolev space $H^{-2}(0,1)$, while the diffusion term lies in $L^2(0,1)$. In contrast to the weighted identity method used in \cite{Gaochenli15}, our results relax the spatial regularity requirements on both the drift and diffusion terms. This improvement will be crucial for deriving a Carleman estimate for the coupled system \eqref{ADJSO1}, particularly to handle the third-order derivatives of $p$ and $q$ and the $L^2$ regularity of the diffusion term in the second equation.

To this end, we first introduce suitable weight functions based on the following result (\cite{BFurIman}): If $\mathcal{O}\cap\mathcal{O}_d^0\neq\emptyset$, then for any nonempty set $\mathcal{B}\Subset\mathcal{O}\cap\mathcal{O}_d^0$, there exists a function $\kappa\in C^\infty([0,1])$ such that
\[
\kappa>0 \quad \text{in } (0,1), \qquad \kappa(0)=\kappa(1)=0, \qquad \|\kappa\|_{C([0,1])}=1, \qquad |\kappa_x|>0 \quad \text{in } [0,1]\setminus \mathcal{B},
\]
\[
\kappa_x(0)>0, \qquad \kappa_x(1)<0.
\]
For large parameters $\lambda,\mu \geq 1$, we choose the following weight functions. For any $(t,x)\in Q$,
\begin{align}\label{weightfunct}
\begin{aligned}
\varphi=\varphi(x)&=\exp(5\mu)-\exp(\mu(\kappa(x)+3)),\qquad\gamma=\gamma(t)=\frac{1}{t(T-t)},\\
&\alpha\equiv\alpha(t,x) = \varphi(x)\gamma(t),\qquad\theta=\exp(-\lambda\alpha).
\end{aligned}
\end{align}
It is easy to verify that there exists a constant $C=C(\mathcal{B})>0$ such that for any $s>0$,
\begin{align}\label{aligned123}
\begin{aligned}
&\,\gamma^{-s}\leq CT^{2s},\quad\qquad\vert\gamma_t\vert\leq CT\gamma^2,\quad\qquad\vert\gamma_{tt}\vert\leq CT^2\gamma^3,\\
&\vert\alpha_t\vert\leq CTe^{5\mu}\gamma^2,\quad\qquad\vert\alpha_{tt}\vert\leq CT^2e^{5\mu}\gamma^3.
\end{aligned}\end{align}

\subsection{Forward Case}
We consider the following forward stochastic fourth-order parabolic equation: 
\begin{equation}\label{eqqgfrse}
\begin{cases}
\begin{array}{ll}
dz +\eta\,z_{xxxx} \,dt = (F_1+F_{2,x}+F_{3,xx}) \,dt + F_4\,dW(t),&\quad (t,x)\in Q,\\
z(t,0)=z(t,1)=0,&\quad t\in(0,T),\\
z_x(t,0)=z_x(t,1)=0,&\quad t\in(0,T),\\
z(0,x)=z_0(x),&\quad x\in(0,1),
\end{array}
\end{cases}
\end{equation}
where $z_0\in L^2_{\mathcal{F}_0}(\Omega;L^2(0,1))$ is the initial state, $F_i\in L^2_\mathcal{F}(0,T;L^2(0,1))$ ($1\leq i\leq4$) source terms. We prove the following new Carleman estimate. 
\begin{thm}\label{themmain5.1}
There exist a large $\mu_1\geq1$ such that for $\mu=\mu_1$, one can find constants $C>0$ and $\lambda_1\geq1$ depending only on $\mathcal{B}$ and $\mu_1$ such that for all $F_i\in L^2_\mathcal{F}(0,T;L^2(0,1))$ and $z_0\in L^2_{\mathcal{F}_0}(\Omega;L^2(0,1))$, the solution $z\in L^2_\mathcal{F}(\Omega;C([0,T];L^2(0,1)))\bigcap L^2_\mathcal{F}(0,T;H^2_0(0,1))$ of \eqref{eqqgfrse} satisfies that
\begin{align}\label{carfor5.6}
\begin{aligned}
&\lambda^7\mathbb{E}\iint_Q \theta^2\gamma^7|z|^2 \,dx \,dt +\lambda^5\mathbb{E}\iint_Q \theta^2\gamma^5|z_x|^2 \,dx \,dt +\lambda^3\mathbb{E}\iint_Q \theta^2\gamma^3|z_{xx}|^2 \,dx \,dt\\
&\leq C \bigg[\lambda^{7}\mathbb{E}\int_0^T\int_\mathcal{B} \theta^{2}\gamma^{7} |z|^2 \,dx\,dt+ \mathbb{E}\iint_Q \theta^{2}|F_1|^2 \,dx\,dt+ \lambda^{2}\mathbb{E}\iint_Q \theta^{2}\gamma^{2}|F_2|^2 \,dx\,dt\\
&\hspace{1cm}+\lambda^{4}\mathbb{E}\iint_Q \theta^{2}\gamma^{4} |F_3|^2 \,dx\,dt+\lambda^{4}\mathbb{E}\iint_Q \theta^{2}\gamma^{4} |F_4|^2 \,dx\,dt\bigg], 
\end{aligned}\end{align}
for any $\lambda\geq\lambda_1(T+T^2)$.
\end{thm}

Recalling the Carleman estimate for the one-dimensional deterministic fourth-order parabolic equation established in \cite[Theorem 1.1]{zhouZ12} for the operator ``\(\partial_t + \partial_x^4\)'', we deduce, by a simple time rescaling \(t \mapsto \eta t\), the following global Carleman estimate for \eqref{eqqgfrse} in the case where \(F_2 \equiv F_3 \equiv F_4 \equiv 0\).
\begin{lm}\label{lemma5.1}
There exist a large $\mu_0\geq1$ such that for $\mu=\mu_0$, one can find constants $C>0$ and $\lambda_0\geq1$ depending only on $\mathcal{B}$ and $\mu_0$ such that for all $F_1\in L^2_\mathcal{F}(0,T;L^2(0,1))$ and $z_0\in L^2_{\mathcal{F}_0}(\Omega;L^2(0,1))$, the solution $z\in L^2_\mathcal{F}(\Omega;C([0,T];L^2(0,1)))\bigcap L^2_\mathcal{F}(0,T;H^2_0(0,1))$ of \eqref{eqqgfrse} with \(F_2 \equiv F_3 \equiv F_4 \equiv0\) satisfies that
\begin{align}\label{carfor5.6se2}
\begin{aligned}
\mathbb{E}\iint_Q \theta^2\Big[\lambda^7\gamma^7|z|^2 +\lambda^5\gamma^5|z_x|^2 +\lambda^3\gamma^3|z_{xx}|^2\Big] \,dx \,dt\leq C\,\mathbb{E}\iint_Q \theta^2\Big[\lambda^7\gamma^7 |z|^2\chi_{\mathcal{B}} +|F_1|^2\Big] \,dx\,dt,
\end{aligned}\end{align}
for any $\lambda\geq\lambda_0(T+T^2)$.
\end{lm}

The proof of the main Carleman estimate \eqref{carfor5.6} is based on the duality method. 
To this end, we introduce the following controlled backward stochastic fourth-order parabolic equation:
\begin{equation}\label{eqqgfrsecon}
\begin{cases}
\begin{array}{ll}
dy - \eta\,y_{xxxx} \, dt = \left[\lambda^7  \theta^2 \gamma^7 z + f\chi_{\mathcal{B}}  \right] \, dt + Y \, dW(t),&\quad (t,x)\in Q,\\
y(t,0)=y(t,1)=0,&\quad t\in(0,T),\\
y_x(t,0)=y_x(t,1)=0,&\quad t\in(0,T),\\
y(T,x)=0,&\quad x\in(0,1),
\end{array}
\end{cases}
\end{equation}
where $(y,Y)$ are the state variables, $f$ is the control variable, $z$ is the solution to equation \eqref{eqqgfrse}, and $\theta$ and $\gamma$ are the weight functions defined in \eqref{weightfunct}. Throughout this subsection, we take $\mu=\mu_0$ as given in Lemma \ref{lemma5.1}. Then we have the following controllability result for \eqref{eqqgfrsecon}.
\begin{prop}\label{prop2.4}
There exists a control $\widehat{f} \in L^2_\mathcal{F}(0,T;L^2(\mathcal{B}))$ such that the associated solution $(\widehat{y},\widehat{Y})$ of \eqref{eqqgfrsecon} satisfies $\widehat{y}(0,\cdot)=0$ in $(0,1)$, a.s. Moreover, there exists a constant $C>0$ such that, for $\mu=C$ and for all $\lambda \ge C(T+T^2)$,
\begin{align}\label{2.2ineq}
\begin{aligned}
&\mathbb{E}\iint_Q \theta^{-2}\Big[\lambda^{-7}\gamma^{-7} |\widehat{f}|^2 +|\widehat{y}|^2+\lambda^{-2}\gamma^{-2}|\widehat{y}_x|^2+\lambda^{-4}\gamma^{-4} |\widehat{y}_{xx}|^2+\lambda^{-4}\gamma^{-4} |\widehat{Y}|^2\Big] \,dx\,dt\\
&\quad\leq C\;\lambda^7\mathbb{E}\iint_Q \theta^2\gamma^7 |z|^2\,dx\,dt.
\end{aligned}
\end{align}
\end{prop}

\begin{proof}
For any $\varepsilon > 0$, define $\theta_\varepsilon = e^{-\lambda \alpha_\varepsilon}$, where $\alpha_\varepsilon \equiv \alpha_\varepsilon(t, x) = \varphi(x)\,[(t + \varepsilon)(T - t + \varepsilon)]^{-1}$. We then consider the following optimal control problem:
\begin{equation}\label{2.03}
    \inf_{f \in \mathcal{H}} \,\frac{1}{2} \mathbb{E}\bigg\{ \int_0^T\int_{\mathcal{B}} \lambda^{-7} \theta^{-2} \gamma^{-7} |f|^2 \, dx \, dt +  \iint_Q \theta^{-2}_\varepsilon |y|^2 \, dx \, dt + \frac{1}{\varepsilon}  \int_0^1 |y(0,x)|^2 \, dx\bigg\},
\end{equation}
subject to \eqref{eqqgfrsecon}, where
\begin{align*}
\mathcal{H} = \bigg\{ f \in L^2_\mathcal{F}(0, T; L^2(\mathcal{B})) : \quad\mathbb{E}\int_0^T\int_{\mathcal{B}} \theta^{-2} \gamma^{-7} |f|^2 \, dx \, dt< \infty \bigg\}.
\end{align*}
It is easy to verify that problem \eqref{2.03} admits a unique optimal solution \(f^\varepsilon \in \mathcal{H}\), which is characterized by
\begin{equation}\label{2.003}
 f^\varepsilon = \lambda^7 \theta^2 \gamma^7 r^\varepsilon \chi_{\mathcal{B}} \quad\textnormal{in}\;Q,\quad \textnormal{a.s},
\end{equation}
where $r^\varepsilon$ denotes the solution of the random fourth-order parabolic equation
\begin{equation}\label{2.04}
\begin{cases}
\begin{array}{ll}
dr^\varepsilon + \eta\,r^\varepsilon_{xxxx} \, dt = \theta_\varepsilon^{-2} y^\varepsilon \, dt,&\quad (t,x)\in Q,\\
r^\varepsilon(t,0)=r^\varepsilon(t,1)=0,&\quad t\in(0,T),\\
r^\varepsilon_x(t,0)=r^\varepsilon_x(t,1)=0,&\quad t\in(0,T),\\
r^\varepsilon(0,x)=\frac{1}{\varepsilon} y^\varepsilon(0,x),&\quad x\in(0,1),
\end{array}
\end{cases}
\end{equation}
where \( (y^\varepsilon,Y^\varepsilon) \) is the solution of \eqref{eqqgfrsecon} corresponding to the control \( f^\varepsilon \). 

We establish a uniform estimate for the family of optimal solutions $\{(f^\varepsilon, y^\varepsilon, Y^\varepsilon)\}_{\varepsilon>0}$. Using \eqref{2.003}, \eqref{carfor5.6se2}, and Itô's formula, we then obtain, for $\lambda \ge C(T+T^2)$,
\begin{align*}
\begin{aligned}
& \frac{1}{\varepsilon} \mathbb{E} \int_0^1 |y^\varepsilon(0,x)|^2 \, dx + \lambda^7 \mathbb{E} \int_0^T\int_{\mathcal{B}} \theta^2 \gamma^7 |r^\varepsilon|^2 \, dx \, dt  + \mathbb{E} \iint_Q \theta_\varepsilon^{-2} |y^\varepsilon|^2 \, dx \, dt \\
& \leq \rho\bigg[\lambda^7 \mathbb{E} \int_0^T\int_{\mathcal{B}} \theta^2 \gamma^7 |r^\varepsilon|^2 \, dx \, dt  + \mathbb{E} \iint_Q \theta_\varepsilon^{-2} |y^\varepsilon|^2 \, dx \, dt\bigg]+C(\rho)\lambda^7 \mathbb{E} \iint_Q \theta^2 \gamma^7 |z|^2 \, dx \, dt.
\end{aligned}
\end{align*}
This implies that 
\begin{align}\label{firstesesp}
\begin{aligned}
\frac{1}{\varepsilon} \mathbb{E} \int_0^1 |y^\varepsilon(0,x)|^2 \, dx + \lambda^{-7} \mathbb{E} \iint_Q\theta^{-2} \gamma^{-7} |f^\varepsilon|^2 \, dx \, dt  + \mathbb{E} \iint_Q \theta_\varepsilon^{-2} |y^\varepsilon|^2 \, dx \, dt \leq C\lambda^7 \mathbb{E} \iint_Q \theta^2 \gamma^7 |z|^2 \, dx \, dt.
\end{aligned}
\end{align}
On the other hand, by computing $d\big(\lambda^{-4}\theta_\varepsilon^{-2}\gamma^{-4}|y^\varepsilon|^2\big)$ and using the fact that, for any $\lambda \ge C(T+T^2)$, 
\begin{align*}
    |(\theta_\varepsilon^{-2}\gamma^{-4})_t|\leq C\lambda ^2\theta_\varepsilon^{-2}\gamma^{-2},\quad |(\theta_\varepsilon^{-2}\gamma^{-4})_x|\leq C\lambda\theta_\varepsilon^{-2}\gamma^{-3},\quad|(\theta_\varepsilon^{-2}\gamma^{-4})_{xx}|\leq C\lambda^2\theta_\varepsilon^{-2}\gamma^{-2},
\end{align*}
we find that
\begin{align*}
&\lambda^{-4}\mathbb{E}\iint_Q \theta_\varepsilon^{-2}\gamma^{-4}|Y^\varepsilon|^2 \,dx \,dt+2\lambda^{-4}\mathbb{E}\iint_Q \theta_\varepsilon^{-2}\gamma^{-4} |y^\varepsilon_{xx}|^2 \,dx \,dt\\
&\leq C\lambda^{-2}\mathbb{E}\iint_Q \theta_\varepsilon^{-2}\gamma^{-2}|y^\varepsilon|^2 \,dx \,dt+C\lambda^{-2}\mathbb{E}\iint_Q \theta_\varepsilon^{-2}\gamma^{-2} |y^\varepsilon| |y^\varepsilon_{xx}| \,dx \,dt+C\lambda^{-3}\mathbb{E}\iint_Q \theta_\varepsilon^{-2}\gamma^{-3} |y^\varepsilon_x| |y^\varepsilon_{xx}| \,dx \,dt \\
&\hspace{0.5cm}+C\lambda^3\mathbb{E}\iint_Q \theta_\varepsilon^{-2}\theta^2\gamma^3 |y^\varepsilon| |z| \,dx \,dt +C\lambda^{-4}\mathbb{E}\int_0^T\int_\mathcal{B} \theta_\varepsilon^{-2}\gamma^{-4}|y^\varepsilon| |f^\varepsilon| \,dx \,dt.
\end{align*}
Using Young's inequality and \eqref{firstesesp}, it follows that, for any $\lambda \ge C(T+T^2)$,
\begin{align}\label{estifseespsec}
\begin{aligned}
&\lambda^{-4}\mathbb{E}\iint_Q \theta_\varepsilon^{-2}\gamma^{-4}|Y^\varepsilon|^2 \,dx \,dt+\lambda^{-4}\mathbb{E}\iint_Q \theta_\varepsilon^{-2}\gamma^{-4} |y^\varepsilon_{xx}|^2 \,dx \,dt\\
&\leq \,\widetilde{C}\lambda^{-2}\mathbb{E}\iint_Q \theta_\varepsilon^{-2}\gamma^{-2}|y^\varepsilon_x|^2 \,dx \,dt+C\lambda^{7}\mathbb{E}\iint_Q \theta^2\gamma^{7}|z|^2 \,dx \,dt.
\end{aligned}
\end{align}
By integration by parts, we have that
$$\,\widetilde{C}\lambda^{-2}\mathbb{E}\iint_Q \theta_\varepsilon^{-2}\gamma^{-2}|y^\varepsilon_x|^2 \,dx \,dt=-\,\widetilde{C}\lambda^{-2}\mathbb{E}\iint_Q \Big[(\theta_\varepsilon^{-2}\gamma^{-2})_x y^\varepsilon_x+\theta_\varepsilon^{-2}\gamma^{-2} y^\varepsilon_{xx}\Big] y^\varepsilon \,dx \,dt.$$
Observe that $|(\theta_\varepsilon^{-2}\gamma^{-2})_x| \le C \lambda \theta_\varepsilon^{-2} \gamma^{-1}$.  It then follows that, for any $\rho > 0$,
\begin{align*}
\,\widetilde{C}\lambda^{-2}\mathbb{E}\iint_Q \theta_\varepsilon^{-2}\gamma^{-2}|y^\varepsilon_x|^2 \,dx \,dt\leq&\, \rho\,\widetilde{C}\lambda^{-2}\mathbb{E}\iint_Q \theta_\varepsilon^{-2}\gamma^{-2}|y^\varepsilon_x|^2 \,dx \,dt+C(\rho)\mathbb{E}\iint_Q \theta_\varepsilon^{-2} |y^\varepsilon|^2 \,dx \,dt\\
&+\frac{1}{4}\lambda^{-4}\mathbb{E}\iint_Q \theta_\varepsilon^{-2}\gamma^{-4} |y^\varepsilon_{xx}|^2 \,dx \,dt.
\end{align*}
Setting $\rho = 1/2$, we obtain
\begin{align}\label{estofporzxn}
\,\widetilde{C}\lambda^{-2}\mathbb{E}\iint_Q \theta_\varepsilon^{-2}\gamma^{-2}|y^\varepsilon_x|^2 \,dx \,dt\leq& \,C \,\mathbb{E}\iint_Q \theta_\varepsilon^{-2} |y^\varepsilon|^2 \,dx \,dt+\frac{1}{2}\lambda^{-4}\mathbb{E}\iint_Q \theta_\varepsilon^{-2}\gamma^{-4} |y^\varepsilon_{xx}|^2 \,dx \,dt.
\end{align}
Combining \eqref{estofporzxn}, \eqref{estifseespsec}, and  \eqref{firstesesp}, we deduce that
\begin{align}\label{unifirstesespsec}
\begin{aligned}
& \frac{1}{\varepsilon} \mathbb{E} \int_0^1 |y^\varepsilon(0,x)|^2 \, dx + \lambda^{-7} \mathbb{E} \iint_Q \theta^{-2} \gamma^{-7} |f^\varepsilon|^2 \, dx \, dt + \mathbb{E} \iint_Q \theta_\varepsilon^{-2} |y^\varepsilon|^2 \, dx \, dt\\
&+\lambda^{-2}\mathbb{E}\iint_Q \theta_\varepsilon^{-2}\gamma^{-2}|y^\varepsilon_x|^2 \,dx \,dt+\lambda^{-4}\mathbb{E}\iint_Q \theta_\varepsilon^{-2}\gamma^{-4} |y^\varepsilon_{xx}|^2 \,dx \,dt+\lambda^{-4}\mathbb{E}\iint_Q \theta_\varepsilon^{-2}\gamma^{-4}|Y^\varepsilon|^2 \,dx \,dt \\
& \quad\leq C\lambda^7 \mathbb{E} \iint_Q \theta^2 \gamma^7 |z|^2 \, dx \, dt.
\end{aligned}
\end{align}
It follows that there exists  $(\widehat{f}, \widehat{y}, \widehat{Y}) \in L^2_\mathcal{F}(0,T;L^2(\mathcal{B})) \times L^2_\mathcal{F}(0,T;H^2_0(0,1))\times L^2_\mathcal{F}(0,T;L^2(0,1)),$ such that as \( \varepsilon \to 0 \),
\begin{align}\label{2.020}
\begin{aligned}
& f^\varepsilon \rightharpoonup \widehat{f}, \quad \textnormal{in} \quad L^2((0,T) \times \Omega; L^2(\mathcal{B})); \\
& y^\varepsilon \rightharpoonup \widehat{y}, \quad \textnormal{in} \quad L^2((0,T) \times \Omega; H^2_0(0,1));\\
& Y^\varepsilon \rightharpoonup \widehat{Y}, \quad \textnormal{in} \quad L^2((0,T) \times \Omega; L^2(0,1)).
\end{aligned}
\end{align}
It is straightforward to verify that $(\widehat{y},\widehat{Y})$ is the solution of \eqref{eqqgfrsecon} corresponding to the control $\widehat{f}$.  
Combining \eqref{unifirstesespsec} with \eqref{2.020} then establishes the null controllability of \eqref{eqqgfrsecon} together with the estimate \eqref{2.2ineq}.
\end{proof}

We now provide the proof of Theorem \ref{themmain5.1}.
\begin{proof}[Proof of Theorem \ref{themmain5.1}]
Let $z$ be the solution of \eqref{eqqgfrse}, and let $(\widehat{y},\widehat{Y})$ be the solution of \eqref{eqqgfrsecon} associated with the control $\widehat{f}$, as obtained in Proposition \ref{prop2.4}.  
Applying Itô's formula then gives
$$\lambda^7\mathbb{E}\iint_Q \theta^2\gamma^7 |z|^2 \,dx\,dt =-\mathbb{E}\iint_Q (F_1\widehat{y}-F_2\widehat{y}_x+F_3\widehat{y}_{xx}+ z\widehat{f}\chi_\mathcal{B}+F_4\widehat{Y}) \,dx\,dt.$$
Hence, for any $\varepsilon > 0$,
\begin{align*}
\begin{aligned}
\lambda^7\mathbb{E}\iint_Q \theta^2\gamma^7 |z|^2 \,dx\,dt\leq&\,\varepsilon\,\mathbb{E}\iint_Q \theta^{-2}\Big[\lambda^{-7}\gamma^{-7} |\widehat{f}|^2 +|\widehat{y}|^2+ \lambda^{-2}\gamma^{-2}|\widehat{y}_x|^2+\lambda^{-4}\gamma^{-4} |\widehat{y}_{xx}|^2+\lambda^{-4}\gamma^{-4} |\widehat{Y}|^2\Big] \,dx\,dt\\
&+\frac{C}{\varepsilon}\,\mathbb{E}\iint_Q \theta^{2}\Big[\lambda^{7}\gamma^{7} |z|^2\chi_{\mathcal{B}} +|F_1|^2 + \lambda^{2}\gamma^{2}|F_2|^2 +\lambda^{4}\gamma^{4} |F_3|^2 +\lambda^{4}\gamma^{4} |F_4|^2 \Big]\,dx\,dt.
    \end{aligned}
\end{align*}
By the estimate \eqref{2.2ineq}, it follows that, for any $\lambda \ge C(T+T^2)$,
\begin{align}\label{firstineforz}
\begin{aligned}
\lambda^7\mathbb{E}\iint_Q \theta^2\gamma^7 |z|^2 \,dx\,dt\leq C\,\mathbb{E}\iint_Q \theta^{2}\Big[\lambda^{7}\gamma^{7} |z|^2\chi_{\mathcal{B}} +|F_1|^2 + \lambda^{2}\gamma^{2}|F_2|^2 +\lambda^{4}\gamma^{4} |F_3|^2 +\lambda^{4}\gamma^{4} |F_4|^2 \Big]\,dx\,dt.
\end{aligned}
\end{align}
On the other hand, applying Itô's formula to $d(\lambda^3 \theta^2 \gamma^3 z^2)$ and using the fact that, for any $\lambda \ge C(T+T^2)$,
\begin{align*}
    |(\theta^2\gamma^3)_t|\leq C\lambda^2\theta^2\gamma^5,\quad |(\theta^2\gamma^3)_x|\leq C\lambda\theta^2\gamma^4,\quad |(\theta^2\gamma^3)_{xx}|\leq C\lambda^2\theta^2\gamma^5,
\end{align*}
we obtain
\begin{align*}
    \begin{aligned}
 2\lambda^3\mathbb{E}\iint_Q \theta^2\gamma^3 |z_{xx}|^2 \,dx\,dt\leq&\,C\lambda^5\mathbb{E}\iint_Q \theta^2\gamma^5 |z|^2 \,dx\,dt+C\lambda^5\mathbb{E}\iint_Q \theta^2\gamma^5 |z| |z_{xx}| \,dx\,dt\\
 &+C\lambda^4\mathbb{E}\iint_Q \theta^2\gamma^4 |z_x| |z_{xx}| \,dx\,dt+2\lambda^3\mathbb{E}\iint_Q \theta^2\gamma^3 |z| |F_1| \,dx\,dt\\
 &+\lambda^3\mathbb{E}\iint_Q \theta^2\gamma^3 |F_4|^2 \,dx\,dt+C\lambda^4\mathbb{E}\iint_Q \theta^2\gamma^4 |z| |F_2| \,dx\,dt\\
 &+C\lambda^3\mathbb{E}\iint_Q \theta^2\gamma^3 |z_x| |F_2| \,dx\,dt+C\lambda^5\mathbb{E}\iint_Q \theta^2\gamma^5 |z| |F_3| \,dx\,dt\\
 &+C\lambda^4\mathbb{E}\iint_Q \theta^2\gamma^4 |z_x| |F_3| \,dx\,dt+C\lambda^3\mathbb{E}\iint_Q \theta^2\gamma^3 |z_{xx}| |F_3| \,dx\,dt.
    \end{aligned}
\end{align*}
Then, for sufficiently large $\lambda \ge C(T+T^2)$,  we get
\begin{align}\label{ineqq511}
    \begin{aligned}
 \lambda^3\mathbb{E}\iint_Q \theta^2\gamma^3 |z_{xx}|^2 \,dx\,dt\leq&\,C\lambda^7\mathbb{E}\iint_Q \theta^2\gamma^7 |z|^2 \,dx\,dt+C\lambda^5\mathbb{E}\iint_Q \theta^2\gamma^5 |z_x|^2 \,dx\,dt\\
 &+C \,\mathbb{E}\iint_Q \theta^2(|F_1|^2+|F_2|^2) \,dx\,dt+C\lambda^3\mathbb{E}\iint_Q \theta^2\gamma^3 (|F_3|^2+|F_4|^2) \,dx\,dt.
    \end{aligned}
\end{align}
Combining \eqref{ineqq511} and \eqref{firstineforz}, and taking $\lambda \ge C(T+T^2)$ sufficiently large, one obtains
\begin{align}\label{ineqq511sec}
    \begin{aligned}
 \lambda^3\mathbb{E}\iint_Q \theta^2\gamma^3 |z_{xx}|^2 \,dx\,dt\leq\,&C\,\mathbb{E}\iint_Q \theta^{2}\Big[\lambda^{7}\gamma^{7} |z|^2\chi_{\mathcal{B}} +|F_1|^2 + \lambda^{2}\gamma^{2}|F_2|^2 +\lambda^{4}\gamma^{4} |F_3|^2 +\lambda^{4}\gamma^{4} |F_4|^2 \Big]\,dx\,dt\\
&+\widetilde{C}\lambda^5\mathbb{E}\iint_Q \theta^2\gamma^5 |z_x|^2 \,dx\,dt.
    \end{aligned}
\end{align}
For the term involving $z_x$, we have that
\begin{align*}
\widetilde{C}\lambda^5\mathbb{E}\iint_Q \theta^2\gamma^5 |z_x|^2 \,dx\,dt=-\,\widetilde{C}\lambda^5\mathbb{E}\iint_Q (\theta^2\gamma^5)_x z_x z\, \,dx\,dt-\,\widetilde{C}\lambda^5\mathbb{E}\iint_Q \theta^2\gamma^5 z_{xx} z\, \,dx\,dt.
\end{align*}
Note that $|(\theta^2 \gamma^5)_x| \le C \lambda \theta^2 \gamma^6$.  
Thus, for any $\varepsilon > 0$,
\begin{align*}
\,\widetilde{C}\lambda^5\mathbb{E}\iint_Q \theta^2\gamma^5 |z_x|^2 \,dx\,dt\leq &\,\varepsilon\,\widetilde{C}\lambda^5\mathbb{E}\iint_Q \theta^2\gamma^5 |z_x|^2 \,dx\,dt+\frac{C}{\varepsilon}\lambda^7\mathbb{E}\iint_Q \theta^2\gamma^7 |z|^2\, \,dx\,dt\\
&+\frac{1}{4}\lambda^3\mathbb{E}\iint_Q \theta^2\gamma^3 |z_{xx}|^2 \,dx\,dt.
\end{align*}
Choosing $\varepsilon = 1/2$ and combining the resulting inequality with \eqref{ineqq511sec}, we conclude that
\begin{align}\label{ineqq511sec22}
    \begin{aligned}
&\lambda^5\mathbb{E}\iint_Q \theta^2\gamma^5 |z_x|^2 \,dx\,dt+\lambda^3\mathbb{E}\iint_Q \theta^2\gamma^3 |z_{xx}|^2 \,dx\,dt\\
&\leq C\,\mathbb{E}\iint_Q \theta^{2}\Big[\lambda^{7}\gamma^{7} |z|^2\chi_{\mathcal{B}} +|F_1|^2 + \lambda^{2}\gamma^{2}|F_2|^2 +\lambda^{4}\gamma^{4} |F_3|^2 +\lambda^{4}\gamma^{4} |F_4|^2 \Big]\,dx\,dt.
    \end{aligned}
\end{align}
Finally, combining \eqref{ineqq511sec22} and \eqref{firstineforz} yields the desired Carleman estimate \eqref{carfor5.6}.
\end{proof}

For any real number $d \neq 7$, applying \eqref{carfor5.6} to $(\lambda \gamma)^{\frac{d-7}{2}} z$ instead of $z$ and choosing $\lambda$ sufficiently large allows one to absorb the lower-order terms, yielding the following general Carleman estimate.
\begin{prop}
For any $d\in\mathbb{R}$, there exist a large $\mu_1\geq1$ such that for $\mu=\mu_1$, one can find constants $C>0$ and $\lambda_1\geq1$ depending only on $\mathcal{B}$ and $\mu_1$ such that for all $F_i\in L^2_\mathcal{F}(0,T;L^2(0,1))$ and $z_0\in L^2_{\mathcal{F}_0}(\Omega;L^2(0,1))$, the solution $z\in L^2_\mathcal{F}(\Omega;C([0,T];L^2(0,1)))\bigcap L^2_\mathcal{F}(0,T;H^2_0(0,1))$ of \eqref{eqqgfrse} satisfies that
\begin{align}\label{carfor5.6sec}
\begin{aligned}
&\lambda^d\mathbb{E}\iint_Q \theta^2\gamma^d|z|^2 \,dx \,dt +\lambda^{d-2}\mathbb{E}\iint_Q \theta^2\gamma^{d-2}|z_x|^2 \,dx \,dt +\lambda^{d-4}\mathbb{E}\iint_Q \theta^2\gamma^{d-4}|z_{xx}|^2 \,dx \,dt\\
&\leq C \bigg[\lambda^{d}\mathbb{E}\int_0^T\int_\mathcal{B} \theta^{2}\gamma^{d} |z|^2 \,dx\,dt+ \lambda^{d-7}\mathbb{E}\iint_Q \theta^{2}\gamma^{d-7}|F_1|^2 \,dx\,dt+ \lambda^{d-5}\mathbb{E}\iint_Q \theta^{d-5}\gamma^{2}|F_2|^2 \,dx\,dt\\
&\hspace{1cm}+\lambda^{d-3}\mathbb{E}\iint_Q \theta^{2}\gamma^{d-3} |F_3|^2 \,dx\,dt+\lambda^{d-3}\mathbb{E}\iint_Q \theta^{2}\gamma^{d-3} |F_4|^2 \,dx\,dt\bigg], 
\end{aligned}\end{align}
for any $\lambda\geq\lambda_1(T+T^2)$.
\end{prop}

\begin{rmk}
The Carleman estimate \eqref{carfor5.6sec} is of independent interest. In particular, it can be applied to establish and improve the following controllability results:
\begin{enumerate}

\item The null controllability of the following more general backward stochastic fourth-order parabolic equation:
\begin{equation*}
dy - \eta\,y_{xxxx}\,dt = \Big(\sum_{i=0}^3 a_i\,\partial_x^i y + b\,Y + f\chi_{\mathcal B}\Big)\,dt + Y\,dW(t)\quad \textnormal{in } Q,
\end{equation*}
where $a_i, b \in L^\infty_{\mathcal{F}}(0,T;L^\infty(0,1))$ for $0 \le i \le 2$, $a_3 \in L^\infty_{\mathcal{F}}(0,T; W^{1,\infty}(0,1))$, and $f\in L^2_\mathcal{F}(0,T;L^2(\mathcal{B}))$ is the control function. By a standard duality argument, it suffices to establish the observability inequality (based on \eqref{carfor5.6sec})
\[
\mathbb{E}\|z(T,\cdot)\|_{L^2(0,1)}^2 \leq C_T \, \mathbb{E} \int_0^T \int_{\mathcal{B}} |z|^2\,dx\,dt,
\]
for all solutions to the adjoint system
\begin{equation*}
dz + \eta\,z_{xxxx}\,dt = -\sum_{i=0}^3 (-1)^i \partial_x^i (a_i z)\,dt - b z\,dW(t)\quad \textnormal{in } Q.
\end{equation*}

\item Regarding the null controllability of the backward stochastic KS equation in \cite[Theorem 1.2]{Gaochenli15}, the regularity assumption on the diffusion coefficient $q$ can be relaxed to $q \in L^\infty_{\mathcal{F}}(0,T;L^\infty(0,1))$.

\item For the insensitizing control problem of the forward stochastic KS equation considered in \cite[Theorem 1.2]{luliu25}, the diffusion coefficient $b$ can be assumed to satisfy $b \in L^\infty_{\mathcal{F}}(0,T;L^\infty(0,1))$.
\end{enumerate}
\end{rmk}

\subsection{Backward Case}
Let us consider the following general backward stochastic fourth order parabolic equation 
\begin{equation}\label{eqqgbc}
\begin{cases}
\begin{array}{ll}
dz -\eta\,z_{xxxx} \,dt = (F_1+F_{2,x}+F_{3,xx}) \,dt + Z\,dW(t),&\quad (t,x)\in Q,\\
z(t,0)=z(t,1)=0,&\quad t\in(0,T),\\
z_x(t,0)=z_x(t,1)=0,&\quad t\in(0,T),\\
z(T,x)=z_T(x),&\quad x\in(0,1),
\end{array}
\end{cases}
\end{equation}
where $z_T \in L^2_{\mathcal{F}_T}(\Omega;L^2(0,1))$ is the terminal state, and $F_i \in L^2_\mathcal{F}(0,T;L^2(0,1))$ for $1 \le i \le 3$. We now derive the following new Carleman estimate for \eqref{eqqgbc}.
\begin{thm}\label{thmm3.2p}
There exist a large $\mu_2\geq1$ such that for  $\mu=\mu_2$, one can find constants $C>0$ and $\lambda_2\geq1$ depending only on $\mathcal{B}$, $\mu_2$ and $T$ such that for all $\lambda\geq\lambda_2$, $F_i\in L^2_\mathcal{F}(0,T;L^2(0,1))$ and $z_T\in L^2_{\mathcal{F}_T}(\Omega;L^2(0,1))$, the solution $(z,Z)\in (L^2_\mathcal{F}(\Omega;C([0,T];L^2(0,1)))\bigcap L^2_\mathcal{F}(0,T;H^2_0(0,1)))\times L^2_\mathcal{F}(0,T;L^2(0,1))$ of \eqref{eqqgbc} satisfies that
\begin{align}\label{carback5.8}
\begin{aligned}
&\lambda^7\mathbb{E}\iint_Q \theta^2\gamma^7|z|^2 \,dx \,dt +\lambda^5\mathbb{E}\iint_Q \theta^2\gamma^5|z_x|^2 \,dx \,dt +\lambda^3\mathbb{E}\iint_Q \theta^2\gamma^3|z_{xx}|^2 \,dx \,dt\\
&\;\leq C \bigg[\lambda^7\mathbb{E}\int_0^T\int_{\mathcal{B}} \theta^2\gamma^7 |z|^2 \,dx\,dt+ \mathbb{E}\iint_Q \theta^2|F_1|^2 \,dx\,dt+ \lambda^{2}\mathbb{E}\iint_Q \theta^{2}\gamma^{2}|F_2|^2 \,dx\,dt\\
&\hspace{1cm}+ \lambda^{4}\mathbb{E}\iint_Q \theta^{2}\gamma^{4}|F_3|^2 \,dx\,dt
+\lambda^4\mathbb{E}\iint_Q \theta^2\gamma^4 |Z|^2 \,dx\,dt\bigg]. 
\end{aligned}\end{align}
\end{thm}

We begin by recalling the following Carleman estimate for solutions of \eqref{eqqgbc} in the special case where $F_2 \equiv F_3 \equiv 0$ (see \cite[Theorem 4.1]{Gaochenli15}).
\begin{lm}\label{lemma5.1sec}
There exist a large $\widetilde{\mu}_0\geq1$ such that for  $\mu=\widetilde{\mu}_0$, one can find constants $C>0$ and $\widetilde{\lambda}_0\geq1$ depending only on $\mathcal{B}$, $\widetilde{\mu}_0$ and $T$ such that for all $\lambda\geq\widetilde{\lambda}_0$, $F_1\in L^2_\mathcal{F}(0,T;L^2(0,1))$ and $z_T\in L^2_{\mathcal{F}_T}(\Omega;L^2(0,1))$, the solution $(z,Z)\in (L^2_\mathcal{F}(\Omega;C([0,T];L^2(0,1)))\bigcap L^2_\mathcal{F}(0,T;H^2_0(0,1)))\times L^2_\mathcal{F}(0,T;L^2(0,1))$ of \eqref{eqqgbc} with $F_2\equiv F_3\equiv0$ satisfies 
\begin{align}\label{carback5.8sec}
\begin{aligned}
\mathbb{E}\iint_Q \theta^2\Big[\lambda^7\gamma^7|z|^2  +\lambda^5\gamma^5|z_x|^2 +\lambda^3\gamma^3|z_{xx}|^2\Big] \,dx \,dt\leq C \,\mathbb{E}\iint_Q \theta^2\Big[\lambda^7\gamma^7 |z|^2\chi_{\mathcal{B}} +|F_1|^2 +\lambda^4\gamma^4 |Z|^2 \Big]\,dx\,dt. 
\end{aligned}\end{align}
\end{lm}

Based on the Carleman estimate \eqref{carback5.8sec}, we obtain the following controllability result for the forward stochastic fourth-order parabolic equation:
\begin{equation}\label{eqqgfrseconsecsec}
\begin{cases}
\begin{array}{ll}
dy + \eta\,y_{xxxx} \, dt = \left[\lambda^7 \theta^2 \gamma^7 z + f\chi_{\mathcal{B}} \right] \, dt + g \, dW(t), &\quad (t,x)\in Q,\\
y(t,0)=y(t,1)=0, &\quad t\in(0,T),\\
y_x(t,0)=y_x(t,1)=0, &\quad t\in(0,T),\\
y(0,x)=0, &\quad x\in(0,1),
\end{array}
\end{cases}
\end{equation}
where $y$ is the state variable, $(f,g)$ denotes the control pair, and $(z,Z)$ is the solution to \eqref{eqqgbc}.
\begin{prop}
There exists a control $(\widehat{f},\widehat{g}) \in L^2_\mathcal{F}(0,T;L^2(\mathcal{B}))\times L^2_\mathcal{F}(0,T;L^2(0,1))$ such that the associated solution $\widehat{y}$ of \eqref{eqqgfrseconsecsec} satisfies $\widehat{y}(T,\cdot)=0$ in $(0,1)$, a.s. Moreover, it holds that
\begin{align}\label{2.2ineqsec}
\begin{aligned}
&\mathbb{E}\iint_Q \theta^{-2}\Big[\lambda^{-7}\gamma^{-7} |\widehat{f}|^2+\lambda^{-4}\gamma^{-4} |\widehat{g}|^2 +|\widehat{y}|^2+\lambda^{-2}\gamma^{-2}|\widehat{y}_x|^2+\lambda^{-4}\gamma^{-4} |\widehat{y}_{xx}|^2\Big] \,dx \,dt\\
&\quad\leq C\;\lambda^7\mathbb{E}\iint_Q \theta^2\gamma^7 |z|^2\,dx\,dt.
\end{aligned}
\end{align}
\end{prop}

\begin{proof}
For any $\varepsilon > 0$, we consider the following minimization problem:
\begin{equation}\label{2.03secsec}
\inf_{(f,g) \in \mathcal{H}} \,\frac{1}{2}\,\mathbb{E}\bigg\{
\int_0^T \int_{\mathcal{B}} \lambda^{-7} \theta^{-2} \gamma^{-7} |f|^2 \, dx \, dt
+ \iint_Q \lambda^{-4} \theta^{-2} \gamma^{-4} |g|^2 \, dx \, dt 
+ \iint_Q \theta_{\varepsilon}^{-2} |y|^2 \, dx \, dt 
+ \frac{1}{\varepsilon} \int_0^1 |y(T,x)|^2 \, dx
\bigg\},
\end{equation}
where $\theta_\varepsilon = e^{-\lambda \alpha_\varepsilon}$, with $\alpha_\varepsilon \equiv \alpha_\varepsilon(t,x) = \varphi(x)\,[(t+\varepsilon)(T-t+\varepsilon)]^{-1}$, subject to \eqref{eqqgfrseconsecsec}. The admissible set $\mathcal{H}$ is defined by
\begin{align*}
\mathcal{H} = \bigg\{ (f,g) \in \big(L^2_\mathcal{F}(0,T;L^2(0,1))\big)^2 : \;
\mathbb{E}\int_0^T \int_{\mathcal{B}} \theta^{-2} \gamma^{-7} |f|^2 \, dx \, dt < \infty, \quad\mathbb{E} \iint_Q \theta^{-2} \gamma^{-4} |g|^2 \, dx \, dt < \infty
\bigg\}.
\end{align*}
It is straightforward to verify that the problem \eqref{2.03secsec} admits a unique optimal solution $(f^\varepsilon,g^\varepsilon) \in \mathcal{H}$, which is characterized by
\begin{equation}\label{2.003sec}
(f^\varepsilon,g^\varepsilon) = \big(\lambda^7 \theta^2 \gamma^7 r^\varepsilon \chi_{\mathcal{B}},\, \lambda^4 \theta^2 \gamma^4 R^\varepsilon\big)
\quad \textnormal{in } Q,\ \textnormal{a.s.},
\end{equation}
where $(r^\varepsilon,R^\varepsilon)$ is the solution to the following backward stochastic fourth-order parabolic equation:
\begin{equation}\label{2.04sec}
\begin{cases}
\begin{array}{ll}
dr^\varepsilon - \eta\,r^\varepsilon_{xxxx} \, dt = \theta_\varepsilon^{-2} y^\varepsilon \, dt+R^\varepsilon \, dW(t),&\quad (t,x)\in Q,\\
r^\varepsilon(t,0)=r^\varepsilon(t,1)=0,&\quad t\in(0,T),\\
r^\varepsilon_x(t,0)=r^\varepsilon_x(t,1)=0,&\quad t\in(0,T),\\
r^\varepsilon(T,x)=-\frac{1}{\varepsilon} y^\varepsilon(T,x),&\quad x\in(0,1),
\end{array}
\end{cases}
\end{equation}
and $y^\varepsilon$ denotes the solution of \eqref{eqqgfrseconsecsec} corresponding to the control $(f^\varepsilon,g^\varepsilon)$. By combining \eqref{2.003sec}, \eqref{carback5.8sec}, and Itô's formula, we obtain the following uniform estimate for the family of optimal solutions $\{(f^\varepsilon, g^\varepsilon, y^\varepsilon)\}_{\varepsilon>0}$,
\begin{align}\label{firstesespsec}
\begin{aligned}
&\frac{1}{\varepsilon} \mathbb{E} \int_0^1 |y^\varepsilon(T,x)|^2 \, dx + \lambda^{-7} \mathbb{E} \iint_Q\theta^{-2} \gamma^{-7} |f^\varepsilon|^2 \, dx \, dt\\
&\quad+ \lambda^{-4} \mathbb{E} \iint_Q\theta^{-2} \gamma^{-4} |g^\varepsilon|^2 \, dx \, dt  + \mathbb{E} \iint_Q \theta_\varepsilon^{-2} |y^\varepsilon|^2 \, dx \, dt\\
&\qquad\leq C\lambda^7 \mathbb{E} \iint_Q \theta^2 \gamma^7 |z|^2 \, dx \, dt.
\end{aligned}
\end{align}
On the other hand, by computing $d\big(\lambda^{-4}\theta_\varepsilon^{-2}\gamma^{-4}|y^\varepsilon|^2\big)$ and arguing as in the proof of Proposition \ref{prop2.4}, we obtain
\begin{align}\label{estifseespsecsec}
\begin{aligned}
\lambda^{-4}\mathbb{E}\iint_Q \theta_\varepsilon^{-2}\gamma^{-4} |y^\varepsilon_{xx}|^2 \,dx \,dt\leq \,\widetilde{C}\lambda^{-2}\mathbb{E}\iint_Q \theta_\varepsilon^{-2}\gamma^{-2}|y^\varepsilon_x|^2 \,dx \,dt+C\lambda^{7}\mathbb{E}\iint_Q \theta^2\gamma^{7}|z|^2 \,dx \,dt.
\end{aligned}
\end{align}
By integration by parts, we also have
\begin{align}\label{estofporzxnsec}
\,\widetilde{C}\lambda^{-2}\mathbb{E}\iint_Q \theta_\varepsilon^{-2}\gamma^{-2}|y^\varepsilon_x|^2 \,dx \,dt\leq& \,C \,\mathbb{E}\iint_Q \theta_\varepsilon^{-2} |y^\varepsilon|^2 \,dx \,dt+\frac{1}{2}\lambda^{-4}\mathbb{E}\iint_Q \theta_\varepsilon^{-2}\gamma^{-4} |y^\varepsilon_{xx}|^2 \,dx \,dt.
\end{align}
Combining \eqref{estofporzxnsec}, \eqref{estifseespsecsec}, and \eqref{firstesespsec}, we deduce
\begin{align}\label{unifirstesespsecsec}
\begin{aligned}
& \frac{1}{\varepsilon} \mathbb{E} \int_0^1 |y^\varepsilon(T,x)|^2 \, dx + \lambda^{-7} \mathbb{E} \iint_Q \theta^{-2} \gamma^{-7} |f^\varepsilon|^2 \, dx \, dt +\lambda^{-4}\mathbb{E}\iint_Q \theta_\varepsilon^{-2}\gamma^{-4}|g^\varepsilon|^2 \,dx \,dt\\
&+ \mathbb{E} \iint_Q \theta_\varepsilon^{-2} |y^\varepsilon|^2 \, dx \, dt+\lambda^{-2}\mathbb{E}\iint_Q \theta_\varepsilon^{-2}\gamma^{-2}|y^\varepsilon_x|^2 \,dx \,dt+\lambda^{-4}\mathbb{E}\iint_Q \theta_\varepsilon^{-2}\gamma^{-4} |y^\varepsilon_{xx}|^2 \,dx \,dt \\
& \quad\leq C\lambda^7 \mathbb{E} \iint_Q \theta^2 \gamma^7 |z|^2 \, dx \, dt.
\end{aligned}
\end{align}
It follows that there exists $(\widehat{f}, \widehat{g}, \widehat{y}) \in L^2_\mathcal{F}(0,T;L^2(\mathcal{B})) \times L^2_\mathcal{F}(0,T;L^2(0,1)) \times L^2_\mathcal{F}(0,T;H^2_0(0,1))$ such that, as $\varepsilon \to 0$,
\begin{align}\label{2.020sec}
\begin{aligned}
& f^\varepsilon \rightharpoonup \widehat{f}, \quad \textnormal{in} \quad L^2((0,T) \times \Omega; L^2(\mathcal{B})); \\
& g^\varepsilon \rightharpoonup \widehat{g}, \quad \textnormal{in} \quad L^2((0,T) \times \Omega; L^2(0,1));\\
& y^\varepsilon \rightharpoonup \widehat{y}, \quad \textnormal{in} \quad L^2((0,T) \times \Omega; H^2_0(0,1)).
\end{aligned}
\end{align}
One can easily verify that $\widehat{y}$ solves \eqref{eqqgfrseconsecsec} associated with the control pair $(\widehat{f},\widehat{g})$. By combining \eqref{unifirstesespsecsec} and \eqref{2.020sec}, we finally obtain the null controllability of \eqref{eqqgfrseconsecsec} along with the estimate \eqref{2.2ineqsec}.
\end{proof}

We now provide the proof of Theorem \ref{thmm3.2p}.  
\begin{proof}[Proof of Theorem \ref{thmm3.2p}]
Let $(z,Z)$ be the solution of \eqref{eqqgbc}, and let $\widehat{y}$ be the solution of \eqref{eqqgfrseconsecsec} corresponding to the control pair $(\widehat{f},\widehat{g})$ obtained in Proposition \ref{prop2.4}.  
Applying Itô's formula yields
\[
\lambda^7 \mathbb{E} \iint_Q \theta^2 \gamma^7 |z|^2 \, dx \, dt
= - \mathbb{E} \iint_Q \big(F_1 \widehat{y} - F_2 \widehat{y}_x + F_3 \widehat{y}_{xx} + z \widehat{f} \chi_\mathcal{B} + Z \widehat{g} \big) \, dx \, dt.
\]
Using the estimate \eqref{2.2ineqsec}, we then obtain
\begin{align*}
\lambda^7 \mathbb{E} \iint_Q \theta^2 \gamma^7 |z|^2 \, dx \, dt
\leq C \, \mathbb{E} \iint_Q \theta^2 \Big[\lambda^7\gamma^7 |z|^2\chi_{\mathcal{B}} +|F_1|^2
+ \lambda^2 \gamma^2 |F_2|^2 + \lambda^4  \gamma^4 |F_3|^2  
+ \lambda^4  \gamma^4 |Z|^2\Big] \, dx \, dt.
\end{align*}
Finally, by applying Itô's formula to $d(\lambda^3 \theta^2 \gamma^3 z^2)$ and following the same computations as in the proof of Theorem \ref{themmain5.1}, we deduce the desired Carleman estimate \eqref{carback5.8}.
\end{proof}

\begin{rmk}
It can be readily verified that the Carleman estimates \eqref{carfor5.6} and \eqref{carback5.8} remain valid when $z$ satisfies the boundary condition
\begin{align}\label{newbcxx}
    \begin{cases}
        z(t,0)=z(t,1)=0, &\quad t\in(0,T),\\
        z_{xx}(t,0)=z_{xx}(t,1)=0, &\quad t\in(0,T).
    \end{cases}
\end{align}
These estimates can be used to show that the robust Stackelberg controllability result stated in Theorem \ref{th4.1SN} continues to hold under the boundary condition \eqref{newbcxx}.
\end{rmk}

\begin{rmk}
In \cite[Theorem 1.3]{ZhangGaoyua}, the authors establish a Carleman estimate similar to \eqref{carback5.8}. However, the weight functions used therein differ from those introduced in the present paper. Consequently, this estimate cannot be directly used in the next section to derive a Carleman estimate for the coupled forward–backward system \eqref{ADJSO1}.
\end{rmk}
\begin{rmk}
Extending the Carleman estimates \eqref{carfor5.6} and \eqref{carback5.8} to higher-dimensional stochastic fourth-order parabolic equations is considerably more challenging (see, e.g., \cite{qiluwang22}) and will be considered in future work.
\end{rmk}

\section{A Carleman Estimate for Coupled Stochastic KS--KdV Systems}\label{sec3}
In this section, we establish a Carleman estimate for the coupled backward–forward system \eqref{ADJSO1}. We adopt the same notations as in Section \ref{sec3sen}.  In what follows, for any $d \in \mathbb{R}$, we introduce the following quantity to simplify the presentation:
\begin{align*}
\mathcal{I}(d,\cdot)=&\,\lambda^d\mathbb{E}\iint_Q \theta^2\gamma^d|\cdot|^2 \,dx \,dt +\lambda^{d-2}\mathbb{E}\iint_Q \theta^2\gamma^{d-2}|\partial_x \cdot|^2 \,dx \,dt +\lambda^{d-4}\mathbb{E}\iint_Q \theta^2\gamma^{d-4}|\partial_{xx} \cdot|^2 \,dx \,dt.
\end{align*}

Combining  \eqref{carfor5.6} and \eqref{carback5.8}, we prove the following main Carleman estimate.
\begin{lm}\label{thmm5.1} 
Let $(p,P;q)$ be the solution of \eqref{ADJSO1}, and assume that  assumption  \eqref{Assump10} holds. Then there exists a sufficiently large constant $\mu_3 \geq 1$ such that, for $\mu = \mu_3$, one can find constants $C>0$ and $\lambda_3 \geq 1$, depending on \(\mathcal{O}\), \(\mathcal{D}\), \(\mathcal{O}_d^0\), $\mathcal{O}_d^1$, $\mathcal{O}_d^2$, \(T\), $\mu_3$, \(\|a\|_\infty\), \(\|b\|_\infty\). such that for all $\lambda \geq \lambda_3$, it holds that
\begin{align}\label{Carlem5.9}
\begin{aligned}
\mathcal{I}(7,p)+\mathcal{I}(9,q)\leq\,C \bigg[\lambda^{47}  \mathbb{E}\int_0^T\int_{\mathcal{O}} \theta^2 \gamma^{47} |p|^2 \, \,dx\,dt+\lambda^{9} \mathbb{E}\iint_Q\theta^2\gamma^{9}|P|^2\,dx\,dt\bigg].
\end{aligned}
\end{align}
\end{lm}
\begin{proof} For clarity, we divide the proof into three steps.\\
\textbf{Step 1. Notation and preliminary estimates.}\\
Set \( w_m = \theta^2 (\lambda \gamma)^m \) with \( m \in \mathbb{N} \). In view of \eqref{Assump10}, we introduce nonempty open subsets \( \mathcal{B}_1 \), \( \mathcal{B}_2 \), and \( \mathcal{B}_3 \) such that
\begin{align}\label{setsBiO12d}
\mathcal{B}_3 \Subset\mathcal{B}_2 \Subset \mathcal{B}_1 \Subset \mathcal{B}_0 \equiv \mathcal{O} \cap \mathcal{O}_d^0,\quad\textnormal{and}\quad\mathcal{B}_1 \cap(\mathcal{O}_d^1\cup\mathcal{O}_d^2)=\emptyset,
\end{align}
where \( \mathcal{B}_i \Subset \mathcal{B}_{i-1} \) (for \( i=1,2,3 \)) means that \( \overline{\mathcal{B}_i} \subset \mathcal{B}_{i-1} \).

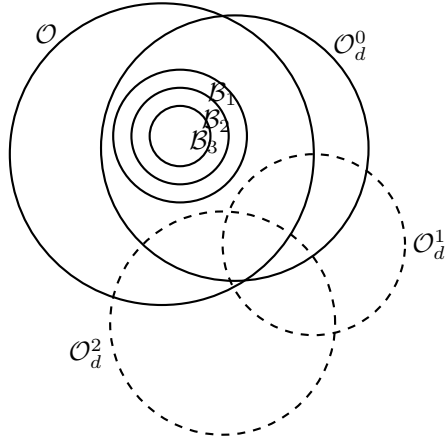
\begin{figure}[H]
\centering
\begin{tikzpicture}[scale=0.8]

\draw[thick] (0,0) circle (2.5);
\node at (-1.9,2.0) {$\mathcal{O}$};

\draw[thick] (1.2,0.1) circle (2.2);
\node at (3.1,1.8) {$\mathcal{O}_d^0$};

\draw[thick, dashed] (2.5,-1.5) circle (1.5);
\node at (4.4,-1.5) {$\mathcal{O}_d^1$};

\draw[thick, dashed] (1.0,-2.8) circle (1.85);
\node at (-1.25,-3.3) {$\mathcal{O}_d^2$};

\draw[thick] (0.3,0.3) circle (1.1);
\node at (1.0,1.0) {$\mathcal{B}_1$};

\draw[thick] (0.3,0.3) circle (0.8);
\node at (0.9,0.55) {$\mathcal{B}_2$};

\draw[thick] (0.3,0.3) circle (0.5);
\node at (0.7,0.2) {$\mathcal{B}_3$};

\end{tikzpicture}
\caption{A configuration of the sets $\mathcal{O}$, $\mathcal{O}_d^0$, $\mathcal{O}_d^1$, and $\mathcal{O}_d^2$, with the subsets $\mathcal{B}_i$ satisfying \eqref{setsBiO12d}.}
\end{figure}

Next, we define the functions \( \zeta_i \in C^{\infty}(\mathbb{R}) \) (for the existence of such functions, see, e.g., \cite{HSP18}), satisfying
\begin{align}\label{assmzeta}
\begin{aligned}
& 0 \leq \zeta_i \leq 1, \quad \zeta_i = 1 \quad \text{in} \quad \mathcal{B}_{4-i}, \quad \text{Supp}(\zeta_i) \subset \mathcal{B}_{3-i}, \\ 
& \frac{(\zeta_i)_{xx}}{\zeta_i^{1/2}} \in L^\infty(G), \quad \frac{(\zeta_i)_x}{\zeta_i^{1/2}} \in L^\infty(G), \quad i = 1, 2,3.
\end{aligned}
\end{align}
From \eqref{aligned123} and \eqref{assmzeta}, we show that for sufficiently large \( \lambda\geq\lambda_0 \), one has
\begin{align}\label{esttmforT}
\begin{aligned}
|\partial_t w_m|\leq C&\lambda^{m+2}\theta^2\gamma^{m+2},\qquad|(w_m\zeta_i)_x|\leq C\lambda^{m+1}\theta^2\gamma^{m+1}\zeta_i^{1/2},\\
&|(w_m\zeta_i)_{xx}|\leq C\lambda^{m+2}\theta^2\gamma^{m+2}\zeta_i^{1/2},\quad i=1,2,3.
\end{aligned}
\end{align}
Firstly, applying the Carleman estimate \eqref{carfor5.6sec} to the solution \(q\) with \(\mathcal{B} \equiv \mathcal{B}_3\) and $d=9$, we conclude that there exists a sufficiently large constant \(\mu_1 \ge 1\) such that, for \(\mu = \mu_1\), one can find a constant \(C>0\) and a sufficiently large \(\lambda_1 \ge 1\) such that, for all \(\lambda \ge \lambda_1\), we have
\begin{align}\label{ineqforIq}
\begin{aligned}
\mathcal{I}(9,q)\leq &\,C \bigg[ \lambda^9 \mathbb{E} \int_0^T \int_{\mathcal{B}_3} \theta^2 \gamma^9 |q|^2 \, dx \, dt +\lambda^6\mathbb{E} \iint_Q \theta^2\gamma^6\bigg|bq-\frac{1}{\delta_2}\,P\bigg|^2 \,dx\,dt\\
&\qquad + \lambda^2\mathbb{E} \iint_Q \theta^2 \gamma^2\left| -k\,q_{xx}+aq +\frac{1}{\beta}p\chi_{\mathcal{D}} -\frac{1}{\delta_1}\,p\right|^2  \, dx \, dt+\lambda^4\mathbb{E} \iint_Q \theta^2\gamma^4|q_{xx}|^2 dxdt\bigg].
\end{aligned}
\end{align}
Taking a large $\lambda\geq\widetilde{\lambda}_1$ in \eqref{ineqforIq}, we end up with
\begin{align}\label{Car4.13}
\begin{aligned}
\mathcal{I}(9,q)\leq C \bigg[  &\lambda^9 \mathbb{E} \int_0^T \int_{\mathcal{B}_3} \theta^2 \gamma^9 |q|^2 \, dx \, dt + \lambda^2\mathbb{E} \iint_Q \theta^2\gamma^2 |p|^2  \, dx \, dt+ \lambda^6\mathbb{E} \iint_Q \theta^2\gamma^6 |P|^2  \, dx \, dt\bigg].
\end{aligned}
\end{align}
Secondly, applying the Carleman estimate \eqref{carback5.8} to the system satisfied by \((p,P)\) with \(\mathcal{B} \equiv \mathcal{B}_3\), we deduce that there exists a constant \(\mu_2 \geq 1\) such that, for \(\mu = \mu_2\), one can find a constant \(C>0\) and a sufficiently large \(\lambda_2 \ge 1\) such that, for all \(\lambda \ge \lambda_2\), the following holds
\begin{align*}
\begin{aligned}
\mathcal{I}(7,p)\leq &\,C \bigg[ \lambda^7 \mathbb{E} \int_0^T \int_{\mathcal{B}_3} \theta^2 \gamma^7 |p|^2 \, dx \, dt +\lambda^4\mathbb{E} \iint_Q \theta^2\gamma^4 |P|^2 \,dx\,dt+\lambda^2\mathbb{E} \iint_Q \theta^2\gamma^2|p_{xx}|^2 dxdt\\
&\qquad+ \mathbb{E} \iint_Q \theta^2 \bigg| k\,p_{xx}-ap-bP +q\chi_{\mathcal{O}_d^0}-q_{xx}\chi_{\mathcal{O}_d^1} \bigg|^2  \, dx \, dt+\lambda^4\mathbb{E} \iint_Q \theta^2\gamma^4|q_{xx}|^2 dxdt\bigg].
\end{aligned}
\end{align*}
Taking a large enough $\lambda\geq\widetilde{\lambda}_2$, it follows that
\begin{align}\label{car4.14}
\begin{aligned}
\mathcal{I}(7,p)\leq C \bigg[ &\lambda^7 \mathbb{E} \int_0^T \int_{\mathcal{B}_3} \theta^2 \gamma^7 |p|^2 \, dx \, dt  + \mathbb{E} \iint_Q \theta^2 |q|^2  \, dx \, dt+\lambda^4\mathbb{E} \iint_Q \theta^2\gamma^4|q_{xx}|^2 dxdt\\
&\;+\lambda^4\mathbb{E} \iint_Q \theta^2\gamma^4 |P|^2 \,dx\,dt\bigg].
\end{aligned}
\end{align}
Combining \eqref{Car4.13} and \eqref{car4.14} and taking \( \mu = \mu_3 = \max(\mu_1, \mu_2) \) and a sufficiently large \( \lambda \geq \max(\lambda_0,\widetilde{\lambda}_1,\widetilde{\lambda}_2) \), we obtain 
\begin{align}\label{firsine1}
\begin{aligned}
\mathcal{I}(7,p)+\mathcal{I}(9,q)\leq C \bigg[ &\lambda^7 \mathbb{E} \int_0^T \int_{\mathcal{B}_3} \theta^2 \gamma^7 |p|^2 \, dx \, dt+\lambda^9 \mathbb{E} \int_0^T \int_{\mathcal{B}_3} \theta^2 \gamma^9 |q|^2 \, dx \, dt +\lambda^6\mathbb{E} \iint_Q \theta^2\gamma^6 |P|^2 \,dx\,dt\bigg].
\end{aligned}
\end{align}
\textbf{Step 2. Elimination of the localized integral term associated with \(q\).}\\
We first note that 
\begin{align}\label{firineqq}
\lambda^9 \mathbb{E}\int_0^T\int_{\mathcal{B}_3} \theta^2\gamma^9 |q|^2 \, dx\,dt\leq \mathbb{E}\iint_Q w_9\zeta_1 |q|^2 \, dx\,dt.
\end{align}
Using Itô's formula for \( d(w_9\zeta_1 pq) \), we conclude that
\begin{align*}
\begin{aligned}
\mathbb{E}\iint_Q w_9\zeta_1 |q|^2 \,dx\,dt=& -\mathbb{E}\iint_Q(w_9\zeta_1)_t pq  \,dx\,dt+\mathbb{E}\iint_Q kw_9\zeta_1 pq_{xx} \,dx\,dt+\mathbb{E}\iint_Q  w_9\zeta_1 pq_{xxx} \,dx\,dt\\
&+\mathbb{E}\iint_Q  w_9\zeta_1 qp_{xxx} \,dx\,dt+\mathbb{E}\iint_Q \eta(w_9\zeta_1p)_{xx} q_{xx} \,dx\,dt+\frac{1}{\delta_1}\mathbb{E}\iint_Q w_9\zeta_1 |p|^2 \,dx\,dt\\
&-\frac{1}{\beta}\mathbb{E}\iint_Q w_9\zeta_1 |p|^2\chi_{\mathcal{D}} \,dx\,dt-\mathbb{E}\iint_Q kw_9\zeta_1 qp_{xx}\,dx\,dt-\mathbb{E}\iint_Q \eta (w_9\zeta_1 q)_{xx}p_{xx}\,dx\,dt\\
&+\mathbb{E}\iint_Q w_9\zeta_1  qq_{xx}\chi_{\mathcal{O}_d^1}\,dx\,dt-\mathbb{E}\iint_Q w_9\zeta_1  qq_{xxxx}\chi_{\mathcal{O}_d^2}\,dx\,dt+\frac{1}{\delta_2}\mathbb{E}\iint_Q w_9\zeta_1 |P|^2 \,dx\,dt.
\end{aligned}
\end{align*}
Fix \(\varepsilon>0\) sufficiently small. By integration by parts, using Young's inequality, \eqref{setsBiO12d} and \eqref{esttmforT}, it follows that, for \(\lambda\) sufficiently large,
\begin{align}\label{ineqqforset2}
\begin{aligned}
\mathbb{E}\iint_Q w_9\zeta_1 |q|^2 \,dx\,dt\leq C\varepsilon \mathcal{I}(9,q)+\frac{C}{\varepsilon}\bigg[ &\lambda^{17} \mathbb{E} \int_0^T \int_{\mathcal{B}_2} \theta^2 \gamma^{17}  |p|^2 \, dx \, dt+\lambda^{15} \mathbb{E} \int_0^T \int_{\mathcal{B}_2} \theta^2 \gamma^{15}  |p_{x}|^2 \, dx \, dt\\
&+\lambda^{13} \mathbb{E} \int_0^T \int_{\mathcal{B}_2} \theta^2 \gamma^{13}  |p_{xx}|^2 \, dx \, dt +\lambda^9\mathbb{E} \iint_Q \theta^2\gamma^9 |P|^2 \,dx\,dt\bigg].
\end{aligned}
\end{align}
Combining \eqref{ineqqforset2} and \eqref{firineqq}, and choosing \(\varepsilon>0\) sufficiently small and \(\lambda\) large enough, we obtain that
\begin{align}\label{firsine1secsec}
\begin{aligned}
\mathcal{I}(7,p)+\mathcal{I}(9,q)\leq C \bigg[ &\lambda^{17} \mathbb{E} \int_0^T \int_{\mathcal{B}_2} \theta^2 \gamma^{17}  |p|^2 \, dx \, dt+\lambda^{15} \mathbb{E} \int_0^T \int_{\mathcal{B}_2} \theta^2 \gamma^{15}  |p_{x}|^2 \, dx \, dt\\
&+\lambda^{13} \mathbb{E} \int_0^T \int_{\mathcal{B}_2} \theta^2 \gamma^{13}  |p_{xx}|^2 \, dx \, dt +\lambda^9\mathbb{E} \iint_Q \theta^2\gamma^9 |P|^2 \,dx\,dt\bigg].
\end{aligned}
\end{align}
\textbf{Step 3. Elimination of the localized integrals associated with \(p_x\) and \(p_{xx}\).}\\
We first eliminate the term associated with \(p_{xx}\). Observe that
\begin{align}\label{pxx11326}
    \lambda^{12} \mathbb{E} \int_0^T \int_{\mathcal{B}_2} \theta^2 \gamma^{12} |p_{xx}|^2 \, dx \, dt \leq \mathbb{E} \iint_Q w_{13} \zeta_2|p_{xx}|^2 \, dx \, dt.
\end{align}
Applying It\^o's formula to \(w_{13}\zeta_2 p^2\), we obtain
\begin{align}\label{eqq3.288}
    \begin{aligned}
2\eta\mathbb{E} \iint_Q w_{13} \zeta_2|p_{xx}|^2 \, dx \, dt = &-\mathbb{E} \iint_Q (w_{13} \zeta_2)_t |p|^2 \,dx\,dt-2\eta\mathbb{E} \iint_Q (w_{13} \zeta_2)_{xx} pp_{xx} \,dx\,dt\\
&-4\eta\mathbb{E} \iint_Q (w_{13} \zeta_2)_x p_x p_{xx} \,dx\,dt-2k\mathbb{E} \iint_Q w_{13} \zeta_2 pp_{xx} \,dx\,dt\\
&+2\mathbb{E} \iint_Q w_{13} \zeta_2 pp_{xxx} \,dx\,dt+2\mathbb{E} \iint_Q w_{13} \zeta_2 ap^2 \,dx\,dt\\
&+2\mathbb{E} \iint_Q w_{13} \zeta_2 bpP \,dx \,dt- \mathbb{E} \iint_Q w_{13} \zeta_2 |P|^2 \,dx\,dt\\
&+2\mathbb{E} \iint_Q w_{13} \zeta_2 p(-q+q_{xx}\chi_{\mathcal{O}_d^1}-q_{xxxx}\chi_{\mathcal{O}_d^2}) \,dx\,dt.
    \end{aligned}
\end{align}
Fix $\varepsilon>0$. Noting that the penultimate term on the right-hand side of \eqref{eqq3.288} is negative, and applying Young's inequality along with \eqref{setsBiO12d}, we obtain, for $\lambda$ sufficiently large,
\begin{align}\label{ineqqforpxx}
    \begin{aligned}
2\eta\mathbb{E} &\iint_Q w_{13} \zeta_2|p_{xx}|^2 \, dx \, dt\\
&\leq C\varepsilon \big[\mathcal{I}(7,p)+\mathcal{I}(9,q)\big]+\frac{C}{\varepsilon}\bigg[ \lambda^{27} \mathbb{E} \int_0^T \int_{\mathcal{B}_1} \theta^2 \gamma^{27}  |p|^2 \, dx \, dt+\lambda^{25} \mathbb{E} \int_0^T \int_{\mathcal{B}_1} \theta^2 \gamma^{25}  |p_{x}|^2 \, dx \, dt\\
&\hspace{4.8cm}+\lambda^{9}\mathbb{E} \iint_Q \theta^2\gamma^{9} |P|^2 \,dx\,dt\bigg].
    \end{aligned}
\end{align}
Combining \eqref{ineqqforpxx}, \eqref{pxx11326}, and \eqref{firsine1secsec}, and choosing \(\varepsilon>0\) sufficiently small and \(\lambda\) large enough, we deduce that
\begin{align}\label{firsine1secsecs2}
\begin{aligned}
\mathcal{I}(7,p)+\mathcal{I}(9,q)\leq C \bigg[ &\lambda^{27} \mathbb{E} \int_0^T \int_{\mathcal{B}_1} \theta^2 \gamma^{27} p^2 \, dx \, dt+\lambda^{25} \mathbb{E} \int_0^T \int_{\mathcal{B}_1} \theta^2 \gamma^{25} |p_{x}|^2 \, dx \, dt\\
&\;+\lambda^{9}\mathbb{E} \iint_Q \theta^2\gamma^{9} |P|^2 \,dx\,dt\bigg].
\end{aligned}
\end{align}
We now eliminate the localized integral term associated with \(p_x\). To this end, we notice
\begin{align}\label{inespx31}
    \lambda^{25} \mathbb{E} \int_0^T \int_{\mathcal{B}_1} \theta^2 \gamma^{25} |p_{x}|^2 \, dx \, dt\leq \mathbb{E} \iint_Q w_{25}\zeta_3 |p_{x}|^2 \, dx \, dt.
\end{align}
On the other hand, applying integration by parts, we have
\begin{align*}
    \mathbb{E} \iint_Q w_{25}\zeta_3 |p_{x}|^2 \, dx \, dt=&- \mathbb{E} \iint_Q (w_{25}\zeta_3p_x)_x p \, dx \, dt\\
    =&- \mathbb{E} \iint_Q (w_{25}\zeta_3)_x p_x p \, dx \, dt- \mathbb{E} \iint_Q w_{25}\zeta_3p_{xx} p \, dx \, dt.
\end{align*}
Hence for any $\varepsilon>0$, it follows that
\begin{align}\label{ineqpx32}
    \mathbb{E} \iint_Q w_{25}\zeta_3 |p_{x}|^2 \, dx \, dt\leq 2\varepsilon \mathcal{I}(7,p)+\frac{C}{\varepsilon}\lambda^{47}\mathbb{E} \int_0^T\int_\mathcal{O} \theta^2\gamma^{47} |p|^2 \,dx\,dt.
\end{align}
Combining \eqref{ineqpx32}, \eqref{inespx31}, and \eqref{firsine1secsecs2}, and choosing \(\varepsilon>0\) sufficiently small and \(\lambda\) sufficiently large, we obtain 
\begin{align}\label{firsine1sfinal}
\begin{aligned}
\mathcal{I}(7,p)+\mathcal{I}(9,q)\leq C \bigg[ &\lambda^{47} \mathbb{E} \int_0^T \int_{\mathcal{O}} \theta^2 \gamma^{47} |p|^2 \, dx \, dt+\lambda^{9}\mathbb{E} \iint_Q \theta^2\gamma^{9} |P|^2 \,dx\,dt\bigg].
\end{aligned}
\end{align}
This completes the proof of Lemma \ref{thmm5.1}.
\end{proof}

\section{Proof of the Main Result}\label{sec4}
This section proves the main result stated in Theorem \ref{th4.1SN}.  To handle the source terms $y_d^i$, $i=0,1,2$, in \eqref{eqq4.7}, we introduce the modified weight functions:
\begin{align}\label{rec1}
&\,\overline{\gamma}\equiv\overline{\gamma}(t) = \frac{1}{\ell(t)},\qquad\overline{\alpha}\equiv\overline{\alpha}(t,x) = \varphi(x)\overline{\gamma}(t),\qquad \overline{\theta}=e^{-\lambda\overline{\alpha}},
\end{align}
where
\begin{equation*}\label{eq:adjoint-system} 
\ell(t)=
\begin{cases}
    \begin{array}{ll}
\frac{T^2}{4} &\quad\textnormal{for}\;\; 0\leq t\leq \frac{T}{2},\\
t(T-t) &\quad\textnormal{for}\;\; \frac{T}{2}\leq t\leq T.
    \end{array}
\end{cases}
\end{equation*}
We also denote by 
$$\overline{\mathcal{I}}(\cdot)=\mathbb{E}\iint_Q \overline{\theta}^2\overline{\gamma}^7|\cdot|^2 \,dx \,dt +\mathbb{E}\iint_Q \overline{\theta}^2\overline{\gamma}^5|\partial_x \cdot|^2 \,dx \,dt +\mathbb{E}\iint_Q \overline{\theta}^2\overline{\gamma}^3|\partial_{xx} \cdot|^2 \,dx \,dt.$$

Let us first establish the following observability inequality.
\begin{prop}\label{Pro4.2}
Assume that  \eqref{Assump10} holds, with \(\beta,\delta_1,\delta_2>0\) sufficiently large. Then there exist a constant \(C_T>0\) depending on \(\mathcal{O}\), \(\mathcal{D}\), \(\mathcal{O}_d^0\), $\mathcal{O}_d^1$, $\mathcal{O}_d^2$, \(T\), \(\|a\|_\infty\), \(\|b\|_\infty\), and a positive weight function \(\rho = \rho(t)\) blowing up as \(t \to T\) such that, for any \(p_T \in L^2_{\mathcal{F}_T}(\Omega; L^2(0,1))\), the solution \((p,P;q)\) of \eqref{ADJSO1} satisfies
\begin{align}\label{observaineq}
\begin{aligned}
&\,\mathbb{E}\|p(0,\cdot)\|^2_{L^2(0,1)}+ \mathbb{E}\iint_Q\rho^{-2}(|q|^2+|q_x|^2+|q_{xx}|^2)\,dx\,dt\leq C_T \, \mathbb{E}\iint_Q \big(|p|^2 \chi_{\mathcal{O}}+ |P|^2\big)\,dx\,dt.
\end{aligned}
\end{align}
\end{prop}
\begin{proof} 
We divide the proof into two steps.\\
\textbf{Step 1.} Let \(\nu \in C^1([0,T])\) be a function such that
\begin{equation*}
\nu = 1 \quad \text{in } \bigg[0,\frac{T}{2}\bigg], \quad 
\nu = 0 \quad \text{in } \bigg[\frac{3T}{4},T\bigg], \quad 
\text{and }\;\; \nu' \le \frac{C}{T^2}.
\end{equation*}
Set \((\phi,\Phi) = \nu (p,P)\). Then \((\phi,\Phi)\) satisfies the backward system
\begin{equation*}
\begin{cases}
\begin{array}{lll}
d\phi -(k\,\phi_{xx}-\phi_{xxx}+\eta\,\phi_{xxxx})\,dt
= \Big[-a\phi-b\Phi+ \nu q\chi_{\mathcal{O}_d^0}-\nu q_{xx}\chi_{\mathcal{O}_d^1}\\
\hspace{6cm}+\nu q_{xxxx}\chi_{\mathcal{O}_d^2}+\nu' p \Big] dt+ \Phi\,dW(t),
& (t,x)\in Q,\\[2mm]
\phi(t,0)=\phi(t,1)=0, & t\in(0,T),\\
\phi_x(t,0)=\phi_x(t,1)=0, & t\in(0,T),\\
\phi(T,x)=0, & x\in(0,1).
\end{array}
\end{cases}
\end{equation*}
Computing $d\phi^2$ and applying Young's inequality, we obtain, for any $t\in(0,T)$,
\begin{align*}
&\mathbb{E}\|\phi(t,\cdot)\|^2_{L^2(0,1)}+\mathbb{E}\int_t^T (\|\phi_x\|^2_{L^2(0,1)}+\|\phi_{xx}\|^2_{L^2(0,1)}+\|\Phi\|^2_{L^2(0,1)}) \,ds\\
    &\leq C\,\mathbb{E}\int_t^T (\|\phi\|^2_{L^2(0,1)}+\|\nu q\|^2_{H^2(0,1)}+\|\nu'p\|^2_{L^2(0,1)}) \,ds.
\end{align*}
Hence, by Gr\"onwall's inequality, we deduce 
\begin{align*}
&\mathbb{E}\|\phi(t,\cdot)\|^2_{L^2(0,1)}+\mathbb{E}\int_t^T (\|\phi\|^2_{L^2(0,1)}+\|\phi_x\|^2_{L^2(0,1)}+\|\phi_{xx}\|^2_{L^2(0,1)}+\|\Phi\|^2_{L^2(0,1)}) \,ds\\
    &\leq C\,\mathbb{E}\int_0^T (\|\nu q\|^2_{H^2(0,1)}+\|\nu'p\|^2_{L^2(0,1)}) \,ds,
\end{align*}
which implies that
\begin{align}\label{energy}
\begin{aligned}
&\,\mathbb{E}\|p(0,\cdot)\|^2_{L^2(0,1)}+\mathbb{E}\int_0^{\frac{T}{2}}\int_0^1 (|p|^2+|p_x|^2+|p_{xx}|^2+|P|^2) \,dx\,dt\\
&\leq C\bigg[\mathbb{E}\int_{\frac{T}{2}}^{\frac{3T}{4}}\int_0^1 |p|^2 \,dx\,dt+\mathbb{E}\int_{0}^{\frac{3T}{4}}\int_0^1 (|q|^2+|q_{x}|^2+|q_{xx}|^2) \,dx\,dt \bigg].
\end{aligned}
\end{align}
Since  \(\overline{\theta}\) and \(\overline{\gamma}\) (resp., \(\theta\) and \(\gamma\)) are bounded in \((0,\frac{T}{2}) \times (0,1)\) (resp., \((\frac{T}{2},\frac{3T}{4}) \times (0,1)\)), we deduce that
\begin{align}\label{energy2}
\begin{aligned}
&\,\mathbb{E}\|p(0,\cdot)\|^2_{L^2(0,1)}+\mathbb{E}\int_0^{\frac{T}{2}}\int_0^1\overline{\theta}^2\overline{\gamma}^7|p|^2  \,dx\,dt\\
&+\mathbb{E}\int_0^{\frac{T}{2}}\int_0^1\overline{\theta}^2\overline{\gamma}^5 |p_x|^2  \,dx\,dt+\mathbb{E}\int_0^{\frac{T}{2}}\int_0^1\overline{\theta}^2\overline{\gamma}^3 |p_{xx}|^2  \,dx\,dt\\ 
&\leq C \bigg[\lambda^7\mathbb{E}\iint_Q\theta^2\gamma^7 p^2\,dx\,dt+  \mathbb{E}\int_{0}^{\frac{3T}{4}}\int_0^1 (|q|^2+|q_{x}|^2+|q_{xx}|^2) \,dx\,dt \bigg].
\end{aligned}
\end{align}
Note that \((\theta,\gamma) = (\overline{\theta},\overline{\gamma})\) on \(\left(\frac{T}{2}, T\right)\). Therefore, we have that
\begin{align}\label{Eq5.18}
\begin{aligned}
&\,\mathbb{E}\int_{\frac{T}{2}}^T\int_0^1\overline{\theta}^2\overline{\gamma}^7|p|^2  \,dx\,dt+\mathbb{E}\int_{\frac{T}{2}}^T\int_0^1\overline{\theta}^2\overline{\gamma}^5 |p_x|^2  \,dx\,dt+\mathbb{E}\int_{\frac{T}{2}}^T\int_0^1\overline{\theta}^2\overline{\gamma}^3 |p_{xx}|^2  \,dx\,dt\\
&\leq \lambda^7\mathbb{E}\iint_Q\theta^2\gamma^7 |p|^2\,dx\,dt+\lambda^5\mathbb{E}\iint_Q \theta^2\gamma^5|p_x|^2\,dx\,dt+\lambda^3\mathbb{E}\iint_Q \theta^2\gamma^3|p_{xx}|^2\,dx\,dt.
\end{aligned}
\end{align}
From \eqref{energy2} and \eqref{Eq5.18}, we conclude that
 \begin{align}\label{Ineq1}
 \begin{aligned}
&\,\mathbb{E}\|p(0,\cdot)\|^2_{L^2(0,1)}+\mathbb{E}\iint_Q\overline{\theta}^2\overline{\gamma}^7|p|^2  \,dx\,dt\\
&+\mathbb{E}\iint_Q\overline{\theta}^2\overline{\gamma}^5 |p_x|^2  \,dx\,dt+\mathbb{E}\iint_Q\overline{\theta}^2\overline{\gamma}^3 |p_{xx}|^2  \,dx\,dt\\
&\leq C \bigg[\lambda^7\mathbb{E}\iint_Q\theta^2\gamma^7 |p|^2\,dx\,dt+\lambda^5\mathbb{E}\iint_Q \theta^2\gamma^5|p_x|^2\,dx\,dt\\
&\hspace{1cm}+\lambda^3\mathbb{E}\iint_Q \theta^2\gamma^3|p_{xx}|^2\,dx\,dt+  \mathbb{E}\int_{0}^{\frac{3T}{4}}\int_0^1 (|q|^2+|q_{x}|^2+|q_{xx}|^2)\,dx\,dt\bigg].
\end{aligned}
\end{align}
Using the system satisfied by \(q\) in \eqref{ADJSO1}, we readily deduce that
\begin{align}\label{Ineq2}
\mathbb{E}\int_{0}^{\frac{3T}{4}}\int_0^1 (|q|^2+|q_{x}|^2+|q_{xx}|^2) \,dx\,dt \leq \bigg(\frac{C}{\beta^2}+\frac{C}{\delta_1^2}\bigg)\,\mathbb{E}\iint_Q\overline{\theta}^2\overline{\gamma}^7|p|^2  \,dx\,dt+\frac{C}{\delta_2^2} \,\mathbb{E}\iint_Q |P|^2  \,dx\,dt.
\end{align}
Combining \eqref{Ineq1} and \eqref{Ineq2}, and choosing $\beta$ and $\delta_1$ sufficiently large, we obtain
\begin{align*}
\mathbb{E}\|p(0,\cdot)\|^2_{L^2(0,1)}+\overline{\mathcal{I}}(p)\leq C \,\mathcal{I}(7,p)+C\,\mathbb{E}\iint_Q |P|^2  \,dx\,dt.
\end{align*}
Using Lemma \ref{thmm5.1} with fixed \(\mu=\mu_3\) and \(\lambda=\lambda_3\), it follows that
\begin{align}\label{estimforOI}
\begin{aligned}
\mathbb{E}\|p(0,\cdot)\|^2_{L^2(0,1)}+\overline{\mathcal{I}}(p)\leq\,C \, \mathbb{E}\iint_Q \big(|p|^2 \chi_{\mathcal{O}}+ |P|^2\big)\,dx\,dt.
\end{aligned}
\end{align}
\textbf{Step 2.}  We choose the weight function \(\rho = \rho(t) = e^{\lambda \overline{\alpha}^{\star}(t)}\), where \(\overline{\alpha}^{\star}(t) = \max_{x \in [0,1]} \overline{\alpha}(t,x)\). Notice that
\begin{align}\label{ito1}
\begin{aligned}
d(\rho^{-2} |q|^2)= -2\rho'\rho^{-3} |q|^2dt+ \rho^{-2} \bigg[&2q \bigg\{\bigg(-k\,q_{xx}- q_{xxx}-\eta q_{xxxx}+aq+\frac{1}{\beta}p\chi_{\mathcal{D}}-\frac{1}{\delta_1}p\bigg) dt \\
&\qquad+\bigg(bq-\frac{1}{\delta_2}P\bigg) dW(t)\bigg\} +\bigg(bq-\frac{1}{\delta_2}P\bigg)^2 dt\bigg].
\end{aligned}
\end{align}
Fix \(t \in (0,T)\). Integrating \eqref{ito1} over \((0,t)\times(0,1)\),  we obtain
\begin{align}\label{inessyoine}
\begin{aligned}
\mathbb{E}\int_0^t\int_0^1 d(\rho^{-2} |q|^2)\, dx=&-2\mathbb{E}\int_0^t\int_0^1\rho'\rho^{-3} |q|^2\, dx \,ds-2k\mathbb{E}\int_0^t\int_0^1 \rho^{-2} qq_{xx} \,dx \,ds\\
&-2\mathbb{E}\int_0^t\int_0^1 \rho^{-2} qq_{xxx} \,dx \,ds- 2\eta\mathbb{E}\int_0^t\int_0^1 \rho^{-2} |q_{xx}|^2 \,dx \,ds\\
&+2\mathbb{E}\int_0^t\int_0^1 a\rho^{-2} |q|^2 \,dx \,ds -\frac{2}{\delta_1}\mathbb{E}\int_0^t\int_0^1 \rho^{-2}pq \,dx  \,ds\\
&  +\frac{2}{\beta}\mathbb{E}\int_0^t\int_0^1 \rho^{-2}pq\chi_{\mathcal{D}} \,dx \,ds+\mathbb{E}\int_0^t\int_0^1 \rho^{-2} \bigg(bq-\frac{1}{\delta_2}P\bigg)^2 \,dx \,ds.
\end{aligned}
\end{align}
Noting that the third term on the right-hand side of \eqref{inessyoine} vanishes by integration by parts, and observing that $2 q_x q_{xx} = (|q_x|^2)_x$. Moreover, the second integral contains a term involving \(q_{xx}\), which can be absorbed by the fourth term. Therefore, we conclude that
\begin{align*}
&\mathbb{E}\int_0^1 \rho^{-2}(t) |q(t)|^2\, dx+\mathbb{E}\int_0^t\int_0^1 \rho^{-2} |q_{xx}|^2\, dx \,ds\\
&\quad\leq C\bigg[\mathbb{E}\int_0^t\int_0^1 \rho^{-2} |q|^2\, dx \,ds+ \mathbb{E}\iint_Q\rho^{-2} |p|^2\, dx \,dt+ \mathbb{E}\iint_Q\rho^{-2} |P|^2\, dx \,dt\bigg].
\end{align*}
Since \(q=q(t,\cdot) \in H^2_0(0,1)\), by the Poincaré inequality and Gr\"onwall's lemma, we deduce that
\begin{align}\label{ineqwithrho}
\mathbb{E}\iint_Q \rho^{-2} (|q|^2+|q_{x}|^2+|q_{xx}|^2)\, dx \,dt\leq C\,\bigg[\mathbb{E}\iint_Q\rho^{-2} |p|^2\, dx \,dt+ \mathbb{E}\iint_Q\rho^{-2} |P|^2\, dx \,dt\bigg].
\end{align}
Observing that \(\rho^{-2} \le \overline{\theta}^2\), it then follows from \eqref{ineqwithrho} and \eqref{rec1} that
\begin{align}\label{esttibar}
\mathbb{E}\iint_Q \rho^{-2} (|q|^2+|q_{x}|^2+|q_{xx}|^2)\, dx \,dt 
\le C \, \bigg[\overline{\mathcal{I}}(p)+ \mathbb{E}\iint_Q |P|^2\, dx \,dt\bigg].
\end{align}
Finally, by combining \eqref{esttibar} and \eqref{estimforOI}, we obtain the desired observability inequality \eqref{observaineq}.
\end{proof} 
We now establish the null controllability of the optimality system \eqref{eqq4.7} and, in doing so, complete the proof of Theorem \ref{th4.1SN}.
\begin{prop}\label{Lm5.5}
Assume that  \eqref{Assump10} holds, that \(\beta,\delta_1,\delta_2>0\) are sufficiently large, and let \(\rho = \rho(t)\) be the weight function given in Proposition \ref{Pro4.2}. Then, for any target function 
\(y_d^i \in L^2_{\mathcal{F}}(0,T;L^2(\mathcal{O}_d^i))\) ($i=0,1,2$) satisfying \eqref{inqAss11SN} and any initial condition 
\(y_0 \in L^2_{\mathcal{F}_0}(\Omega;L^2(0,1))\), there exist controls 
\((\widehat{f},\widehat{g}) \in L^2_{\mathcal{F}}(0,T;L^2(\mathcal{O})) \times L^2_{\mathcal{F}}(0,T;L^2(0,1))\) 
minimizing \(\mathcal{J}\) such that the corresponding solution \((\widehat{y};\widehat{z},\widehat{Z})\) of system \eqref{eqq4.7} satisfies
$$\widehat{y}(T,\cdot) =0\;\;\textnormal{in}\;\;(0,1),\quad\textnormal{a.s.}$$
Moreover, the controls \((\widehat{f},\widehat{g})\) can be chosen so that
\begin{align}\label{estiforcon}
\begin{aligned}
&\|\widehat{f}\|^2_{L^2_{\mathcal{F}}(0,T;L^2(\mathcal{O}))} 
+ \|\widehat{g}\|^2_{L^2_{\mathcal{F}}(0,T;L^2(0,1))} \\
&\;\le C_T \bigg[ \mathbb{E}\|y_0\|^2_{L^2(0,1)} 
+ \mathbb{E}\iint_Q \rho^2 \Big(\big|y_d^0\big|^2\chi_{\mathcal{O}_d^0}+\big|y_d^1\big|^2\chi_{\mathcal{O}_d^1}+\big|y_d^2\big|^2\chi_{\mathcal{O}_d^2}\Big) \,dx\,dt \bigg],
\end{aligned}
\end{align}
where \(C_T > 0\) is the same constant as in \eqref{observaineq}.
\end{prop}
\begin{proof}
Let \(y_0 \in L^2_{\mathcal{F}_0}(\Omega; L^2(0,1))\) and \(\varepsilon > 0\). We consider the following optimal control problem:
\begin{align}\label{05.1}
\inf \Big\{ \mathcal{J}_\varepsilon(f,g) \,\big|\, (f,g) \in L^2_{\mathcal{F}}(0,T;L^2(\mathcal{O})) \times L^2_{\mathcal{F}}(0,T;L^2(0,1)) \Big\},
\end{align}
where
\[
\mathcal{J}_\varepsilon(f,g) = \frac{1}{2} \, \mathbb{E} \int_0^T \int_{\mathcal{O}} |f|^2 \, dx \, dt 
+ \frac{1}{2} \, \mathbb{E} \iint_Q |g|^2 \, dx \, dt 
+ \frac{1}{2\varepsilon} \, \mathbb{E} \int_0^1 |y(T,x)|^2 \, dx,
\]
and \((y; z, Z)\) is the solution of \eqref{eqq4.7} corresponding to the initial state \(y_0\) and controls \(f\) and \(g\). 

It is readily seen that \(\mathcal{J}_\varepsilon\) is well-defined, continuous, strictly convex, and coercive. Therefore, the problem \eqref{05.1} admits a unique solution 
\[
(f^\varepsilon, g^\varepsilon) \in L^2_{\mathcal{F}}(0,T; L^2(\mathcal{O})) \times L^2_{\mathcal{F}}(0,T; L^2(0,1)).
\]

By standard duality arguments, the optimal control pair \((f^\varepsilon, g^\varepsilon)\) can be characterized as
\begin{align}\label{05.2}
(f^\varepsilon,g^\varepsilon) = (p^\varepsilon \, \chi_{\mathcal{O}},P^\varepsilon) \quad \textnormal{in } Q,\ \textnormal{a.s.},
\end{align}
where \((p^\varepsilon, P^\varepsilon; q^\varepsilon)\) is the solution of the following backward–forward stochastic KS--KdV system:
\begin{equation*}
\begin{cases}
\begin{array}{lll}
dp^\varepsilon -(k\,p^\varepsilon_{xx}- p^\varepsilon_{xxx}+\eta\,p^\varepsilon_{xxxx})\,dt
= \big[-ap^\varepsilon-bP^\varepsilon+ q^\varepsilon\chi_{\mathcal{O}_d^0}-q_{xx}^\varepsilon\chi_{\mathcal{O}_d^1}+q_{xxxx}^\varepsilon\chi_{\mathcal{O}_d^2}\big] dt
+ P^\varepsilon\,dW(t),
& (t,x)\in Q,\\[2mm]
dq^\varepsilon +(k\,q^\varepsilon_{xx}+ q^\varepsilon_{xxx}+\eta\,q^\varepsilon_{xxxx})\,dt
= \left[aq+\frac{1}{\beta} p^\varepsilon\chi_{\mathcal{D}}-\frac{1}{\delta_1} p^\varepsilon\right] dt+ \big[bq^\varepsilon-\frac{1}{\delta_2} P^\varepsilon\big]\,dW(t),
& (t,x)\in Q,\\[2mm]
p^\varepsilon(t,0)=p^\varepsilon(t,1)=0, & t\in(0,T),\\
p^\varepsilon_x(t,0)=p^\varepsilon_x(t,1)=0, & t\in(0,T),\\
q^\varepsilon(t,0)=q^\varepsilon(t,1)=0, & t\in(0,T),\\
q^\varepsilon_x(t,0)=q^\varepsilon_x(t,1)=0, & t\in(0,T),\\
p^\varepsilon(T,x)=-\frac{1}{\varepsilon}y^\varepsilon(T,x), & x\in(0,1),\\
q^\varepsilon(0,x)=0, & x\in(0,1),
\end{array}
\end{cases}
\end{equation*}
with  $(y^\varepsilon;z^\varepsilon,Z^\varepsilon)$ is the solution of \eqref{eqq4.7} associated to the controls $f^\varepsilon$ and $g^\varepsilon$.

Applying Itô's formula to the terms $y^\varepsilon p^\varepsilon$ and  $z^\varepsilon q^\varepsilon$, integrating by parts and applying Fubini's theorem, we deduce that
\begin{align}\label{dual}
\begin{aligned}
-\frac{1}{\varepsilon}\mathbb{E}\|y^\varepsilon(T,\cdot)\|^2_{L^2(0,1)}
-\mathbb{E}\left\langle y_0,p^\varepsilon(0,\cdot)\right\rangle_{L^2(0,1)}
=&\,\mathbb{E}\int_0^T\int_{\mathcal{O}} f^\varepsilon p^\varepsilon \, dx \, dt
+\mathbb{E}\iint_{Q} g^\varepsilon P^\varepsilon \, dx \, dt \\
&+\mathbb{E}\int_0^T\int_{\mathcal{O}_d^0} y_d^0 q^\varepsilon \, dx \, dt+\mathbb{E}\int_0^T\int_{\mathcal{O}_d^1} y_d^1 q_x^\varepsilon \, dx \, dt\\
&+\mathbb{E}\int_0^T\int_{\mathcal{O}_d^2} y_d^2 q_{xx}^\varepsilon \, dx \, dt.
\end{aligned}
\end{align}
Recalling \eqref{05.2}, it follows from \eqref{dual} that
\begin{align*}
\begin{aligned}
&\mathbb{E}\int_0^T\int_{\mathcal{O}} |p^\varepsilon|^2  \, dx \,dt +\mathbb{E}\iint_{Q} |P^\varepsilon|^2  \, dx \,dt+\frac{1}{\varepsilon}\mathbb{E}\int_0^1|y^\varepsilon(T,x)|^2 dx\\
&=-\mathbb{E}\left\langle y_0,p^\varepsilon (0,\cdot)\right\rangle_{L^2(0,1)}-\mathbb{E}\int_0^T\int_{\mathcal{O}_d^0} y_d^0 q^\varepsilon \, dx \,dt-\mathbb{E}\int_0^T\int_{\mathcal{O}_d^1} y_d^1 q_x^\varepsilon \, dx \, dt\\
&\quad-\mathbb{E}\int_0^T\int_{\mathcal{O}_d^2} y_d^2 q_{xx}^\varepsilon \, dx \, dt.
\end{aligned}
\end{align*}
This yields
\begin{align}\label{05.5}
\begin{aligned}
&\, \mathbb{E}\int_0^T\int_{\mathcal{O}} |p^\varepsilon|^2  \, dx \,dt +\mathbb{E}\iint_{Q} |P^\varepsilon|^2  \, dx \,dt+\frac{1}{\varepsilon}\mathbb{E}\int_0^1|y^\varepsilon(T,x)|^2 dx \\
&\leq \frac{C_T}{2} \bigg[\mathbb{E}\| y_0 \|^2_{L^2(0,1)}+\mathbb{E}\iint_Q \rho^2 \Big(\big|y_d^0\big|^2\chi_{\mathcal{O}_d^0}+\big|y_d^1\big|^2\chi_{\mathcal{O}_d^1}+\big|y_d^2\big|^2\chi_{\mathcal{O}_d^2}\Big) \,dx\,dt\bigg] \\
&\qquad+ \frac{1}{2C_T} \bigg[\mathbb{E}\| p^\varepsilon(0,\cdot) \|^2_{L^2(0,1)}+\mathbb{E}\iint_Q \rho^{-2} (|q^\varepsilon|^2+|q_x^\varepsilon|^2+|q_{xx}^\varepsilon|^2) \,dx\,dt\bigg],
\end{aligned}
\end{align}
where $C_T$ is the observability constant appearing in \eqref{observaineq}.

Combining \eqref{05.5} and \eqref{observaineq}, and using \eqref{05.2}, we deduce that
\begin{align}\label{05.6}
\begin{aligned}
&\|f^\varepsilon\|^2_{L^2_\mathcal{F}(0,T;L^2(\mathcal{O}))}
+\|g^\varepsilon\|^2_{L^2_\mathcal{F}(0,T;L^2(0,1))}
+\frac{2}{\varepsilon}\mathbb{E}\int_0^1 |y^\varepsilon(T,x)|^2\,dx  \\
&\le C_T\bigg[\mathbb{E}\|y_0\|^2_{L^2(0,1)}
+\mathbb{E}\iint_Q \rho^2 \Big(\big|y_d^0\big|^2\chi_{\mathcal{O}_d^0}+\big|y_d^1\big|^2\chi_{\mathcal{O}_d^1}+\big|y_d^2\big|^2\chi_{\mathcal{O}_d^2}\Big) \,dx\,dt\bigg].
\end{aligned}
\end{align}
It follows from \eqref{05.6} that there exist subsequences of $(f^\varepsilon)$ and $(g^\varepsilon)$, still denoted by $(f^\varepsilon)$ and $(g^\varepsilon)$, such that as $\varepsilon\to0$,
\begin{align}\label{05.7}
\begin{aligned}
&f^\varepsilon \rightharpoonup \widehat{f}
\quad \text{in } L^2((0,T)\times\Omega;L^2(\mathcal{O})), \\
&g^\varepsilon \rightharpoonup \widehat{g}
\quad \text{in } L^2((0,T)\times\Omega;L^2(0,1)).
\end{aligned}
\end{align}
Let $(\widehat{y};\widehat{z},\widehat{Z})$ be the solution of \eqref{eqq4.7} associated with $y_0$, $\widehat{f}$, and $\widehat{g}$. Then, in view of \eqref{05.6} and \eqref{05.7}, we obtain
\begin{align*}
&\|\widehat{f}\|^2_{L^2_\mathcal{F}(0,T;L^2(\mathcal{O}))}
+\|\widehat{g}\|^2_{L^2_\mathcal{F}(0,T;L^2(0,1))}\\
&\;\le
C_T\bigg[\mathbb{E}\|y_0\|^2_{L^2(0,1)}
+\mathbb{E}\iint_Q \rho^2 \Big(\big|y_d^0\big|^2\chi_{\mathcal{O}_d^0}+\big|y_d^1\big|^2\chi_{\mathcal{O}_d^1}+\big|y_d^2\big|^2\chi_{\mathcal{O}_d^2}\Big) \,dx\,dt\bigg].
\end{align*}
Moreover, it is not difficult to show that
\[
\widehat{y}(T,\cdot)=0 \quad \text{in } (0,1), \quad \text{a.s.}
\]
Finally, one can also check that the control $(\widehat{f},\widehat{g})$ minimizes the functional $\mathcal{J}$ among all controls achieving null controllability. This completes the proof of Proposition \ref{Lm5.5} and hence of Theorem \ref{th4.1SN}.
\end{proof}


\begin{thebibliography}{10}
\bibitem{albarasan23}
I. C. Albuquerque, F. D. Araruna, M. C. Santos. 
\newblock On a multi-objective control problem for the Korteweg–de Vries equation. 
\newblock{\em Calculus of Variations and Partial Differential Equations}, 62 (2023), 131.


\bibitem{ArFeGu17}  	
F. D. Araruna, E. Fernández-Cara, S. Guerrero and M. C. Santos.   
\newblock New results on the Stackelberg-Nash exact control of linear parabolic equations. 
\newblock{\em Systems \& Control Letters}, 104 (2017), 78--85.


\bibitem{ArCaSa15}  
F. D. Araruna, E. Fernandez-Cara and M. C. Santos.
\newblock  Stackelberg-Nash exact controllability for linear and semilinear parabolic equations.   
\newblock{\em ESAIM: Control, Optimisation and Calculus of Variations}, 21 (2015), 835--856.


\bibitem{benny66}
D. Benney. 
\newblock Long waves on liquid films. \newblock{\em Journal of Mathematics and Physics}, 45 (1966), 150--155.


\bibitem{bewl12}
T. Bewley, P. Moin and R. Temam. \newblock Optimal and robust approaches for linear and nonlinear regulation problems in fluid mechanics. 
\newblock{\em In 4th Shear Flow Control Conference}, (1997), p. 1872.


\bibitem{bewl12se}
T. R. Bewley, R. Temam and M. Ziane. \newblock  A general framework for robust control in fluid mechanics. \newblock{\em Physica D: Nonlinear Phenomena}, 138 (2000), 360--392.

\bibitem{bewl1}
T. Bewley, R. Temam and M. Ziane. \newblock Existence and uniqueness of optimal control to the Navier–Stokes equations. 
\newblock{\em Comptes Rendus de l'Académie des Sciences-Series I-Mathematics}, 330 (2000), 1007--1011.


 
\bibitem{BoMaOuNash} 
I. Boutaayamou, L. Maniar and O.  Oukdach.
\newblock Stackelberg-Nash null controllability of heat equation with general dynamic boundary conditions.  \newblock{\em Evolution Equations and Control Theory}, 11 (2022), 1285--1307.

\bibitem{bous77}
J. Boussinesq.
\newblock Essai sur la théorie des eaux courantes, \newblock{\em  Mémoires présentés par divers savants à l’Académie des Sciences de l’Institute National de France}, 23 (1877), 1–680.

\bibitem{refee14}
L. Breton and C. Montoya. 
\newblock Robust Stackelberg controllability for the Kuramoto–Sivashinsky equation. 
\newblock{\em Mathematics of Control, Signals, and Systems}, 34 (2022), 515--558.

\bibitem{Calsavara} 
B. M. Calsavara, E. Fernández-Cara, L. De Teresa and J. A. Villa.  
\newblock New results concerning the hierarchical control of linear and semilinear parabolic equations.
\newblock{\em ESAIM: Control, Optimisation and Calculus of Variations},  28 (2022), 14.






\bibitem{StackKS}
N. Carreno and M. C. Santos. \newblock Stackelberg-Nash exact controllability for the Kuramoto--Sivashinsky equation with boundary and distributed controls. 
\newblock{\em Journal of Differential Equations}, 343 (2023), 1--63.
 


\bibitem{KSStack2} 
N. Carreno and M. C.  Santos. \newblock Stackelberg–Nash exact controllability for the Kuramoto–Sivashinsky equation. 
\newblock{\em Journal of Differential Equations}, 266 (2019), 6068-6108.

\bibitem{KS23e} 
E. Cerpa. 
\newblock Control of a Korteweg-de Vries equation: a tutorial. 
\newblock{\em Mathematical Control and Related Fields}, 4 (2014), 45--99.

\bibitem{KS2}  	
E. Cerpa and A. Mercado. 
\newblock Local exact controllability to the trajectories of the 1-D Kuramoto–Sivashinsky equation. \newblock{\em Journal of Differential Equations}, 250 (2011), 2024--2044.


\bibitem{KS3}  	
E. Cerpa, A. Mercado and A. F. Pazoto. 
\newblock Null controllability of the stabilized Kuramoto--Sivashinsky system with one distributed control. \newblock{\em SIAM Journal on Control and Optimization}, 53 (2015), 1543--1568.





\bibitem{duanervin}
J. Duan and V. J. Ervin. 
\newblock On the stochastic Kuramoto–Sivashinsky equation. 
\newblock{\em Nonlinear Analysis}, 44 (2001), 205–216.




\bibitem{Nashstoch3}
A. Elgrou, L. Maniar and O. Oukdach. \newblock Inverse initial problem under Nash strategy for stochastic reaction-diffusion equations with dynamic boundary conditions. 
\newblock{\em Journal of Inverse and Ill-posed Problems}, 33 (2025), 529--546.

\bibitem{Nashstoch4}
A. Elgrou and O. Oukdach. 
\newblock Stackelberg-Nash controllability for abstract stochastic evolution equations and applications. 
\newblock{\em Journal of Dynamical and Control Systems}, 32 (2026), 7.


\bibitem{eketm99}
I. Ekeland and R. Temam. 
Convex analysis and variational problems. Society for Industrial and Applied Mathematics, 1999.


\bibitem{malfengkaw00133}
B.-F. Feng, B. A. Malomed and T. Kawahara. 
\newblock Stable periodic waves in coupled Kuramoto-Sivashinsky-Korteweg-de Vries equations. 
\newblock{\em Journal of the Physical Society of Japan}, 71 (2002), 2700--2707.




\bibitem{BFurIman} 
A. Y. Fursikov and O.Y. Imanuvilov.
\newblock Controllability of evolution equations.
Lecture Notes Series, Research Institute of Mathematics, Seoul National University, Korea, 34 (1996), 163.




\bibitem{Gaochenli15}
P. Gao, M. Chen and Y. Li. 
\newblock Observability estimates and null controllability for forward and backward linear stochastic Kuramoto--Sivashinsky equations. 
\newblock{\em SIAM Journal on Control and Optimization}, 53 (2015), 475--500.


 
\bibitem{GRP02} 
R. Glowinski, A. M. Ramos and J.  Periaux. 
\newblock Nash equilibria for the multi-objective control of linear partial differential equations. 
\newblock{\em Journal of Optimization Theory and Applications}, 112 (2002), 457--498.



\bibitem{korVries895}
D. J. Korteweg and G. de Vries. 
\newblock On the change of form of long waves advancing in a rectangular canal and on a new type of long stationary waves.
\newblock{\em Philosophical Magazine}, 39 (1895), 422–443.


\bibitem{kur78}
Y. Kuramoto. 
\newblock Diffusion-induced chaos in reaction systems. 
\newblock{\em Progress of Theoretical Physics Supplement}, 64 (1978), 346--367.


\bibitem{kur75tuz}
Y. Kuramoto and T. Tsuzuki. 
\newblock On the formation of dissipative structures in reaction-diffusion systems: Reductive perturbation approach. 
\newblock{\em Progress of Theoretical Physics}, 54 (1975), 687--699.

\bibitem{kur76tuzu}
Y. Kuramoto and T. Tsuzuki. 
\newblock Persistent propagation of concentration waves in dissipative media far from thermal equilibrium. \newblock{\em Progress of Theoretical Physics}, 55 (1976), 356--369.


\bibitem{Liangwu22}
S. Liang and K. N. Wu.
\newblock Boundary control of stochastic Korteweg-de Vries-Burgers equations. \newblock{\em Nonlinear Dynamics}, 108 (2022), 4093--4102.


\bibitem{LiPa} 
 J.-L. Lions.
\newblock Hierarchic control.
 \newblock{\em  Proceedings of the Indian National Science Academy}, 
 104 (1994), 295--304.

\bibitem{lu2011some}
Q. L{\"u}.
\newblock Some results on the controllability of forward stochastic heat equations with control on the drift.
\newblock{\em Journal of Functional Analysis}, 260 (2011), 832--851.

\bibitem{qiluwang22}
Q. Lü and Y. Wang.
\newblock Null controllability for fourth order stochastic parabolic equations.  
\newblock{\em SIAM Journal on Control and Optimization}, 60
(2022), 1563--1590.


\bibitem{luliu25}
Y. Lu and L. Liu. 
\newblock Insensitizing controls for stochastic Kuramoto–Sivashinsky equation. 
\newblock{\em Mathematical Methods in the Applied Sciences}, (2025).


\bibitem{StDf}
G. Leitmann. 
\newblock On generalized Stackelberg strategies. 
\newblock{\em Journal of Optimization Theory and Applications}, 26 (1978), 637--643.

\bibitem{laucuma96}
K. B. Lauritsen, R. Cuerno and H. A. Makse. 
\newblock Noisy Kuramoto--Sivashinsky equation for an erosion model. \newblock{\em Physical Review E}, 54 (1996), 3577.




\bibitem{Ssivan4}
D. M. Michelson and G. I. Sivashinsky.
\newblock Nonlinear analysis of hydrodynamic instability in laminar flames—II. Numerical experiments. \newblock{\em Acta Astronautica}, 4 (1977), 1207--1221.

\bibitem{refee12}
C. Montoya and L. De Teresa. 
\newblock Robust Stackelberg controllability for the Navier–Stokes equations. 
\newblock{\em Nonlinear Differential Equations and Applications NoDEA}, 25 (2018), 46.


\bibitem{Na51} 
J. Nash.
\newblock Non-cooperative games. 
 \newblock{\em Annals of Mathematics}, 54 (1951), 286--295.




\bibitem{unbounded}
L. L. D. Njoukoué and G. Deugoué. \newblock Stackelberg control in an unbounded domain for a parabolic equation. 
\newblock{\em Journal of Nonlinear Evolution Equations and Applications}, 2021 (2021), 95--118.




\bibitem{Nashstoch1}
O. Oukdach, S. Boulite, A. Elgrou and L. Maniar. 
\newblock Multi-objective control for stochastic parabolic equations with dynamic boundary conditions. 
\newblock{\em Journal of Optimization Theory and Applications}, 204 (2025), 55.

\bibitem{Nashstoch2}
O. Oukdach, S. Boulite, A. Elgrou and L. Maniar. 
\newblock Stackelberg–Nash Null Controllability for Stochastic Parabolic Equations. 
\newblock{\em Mathematical Methods in the Applied Sciences}, 48 (2025), 13164-13176.















\bibitem{BoMaOuNash2}
O. Oukdach, I. Boutaayamou and L.  Maniar.
\newblock Hierarchical control problem for the heat equation with dynamic boundary conditions. 
\newblock{\em IMA Journal of Mathematical Control and Information}, 41 (2024), 255--274.





\bibitem{Pa96} 
V. Pareto. 
\newblock Cours d'économie politique.  Librairie Droz, 1964.


\bibitem{hevicto1}
V. H. Santamaria and L. Peralta. \newblock Controllability results for stochastic coupled systems of fourth-and second-order parabolic equations. \newblock{\em Journal of Evolution Equations}, 22 (2022), 23.

\bibitem{hevicto12}
V. H. Santamaria, K. Le Balc’h and L. Peralta. 
\newblock Statistical null-controllability of stochastic nonlinear parabolic equations. \newblock{\em Stochastics and Partial Differential Equations: Analysis and Computations}, 10 (2022), 190--222.

\bibitem{refee13}
V. H. Santamaria and L. Peralta.
\newblock Some remarks on the Robust Stackelberg controllability for the heat equation with controls on the boundary.
\newblock{\em Discrete and Continuous Dynamical Systems Series B}, 25 (2020), 161--190.

 \bibitem{HSP18} 
V. H. Santamaria and L. De Teresa.
\newblock Some remarks on the hierarchic control for coupled parabolic PDEs, In Recent Advances in PDEs: Analysis, Numerics and Control. Springer, Cham. 17 (2018), 117--137.



 




\bibitem{refee11}
V. H. Santamaria and L. De Teresa.
\newblock Robust Stackelberg controllability for linear and semilinear heat equations. 
\newblock{\em Evolution Equations and Control Theory}, 7 (2018), 247--273.


\bibitem{Tangzhang9}
S. Tang and X. Zhang. 
\newblock Null controllability for forward and backward stochastic parabolic equations. 
\newblock{\em SIAM Journal on Control and Optimization}, 48 (2009), 2191--2216.

 
\bibitem{St34} 
H. V. Stackelberg. 
\newblock Marktform und Gleichgewicht. Springer. 1934.
 


\bibitem{yuzhang23}
Y. Yu and J. F. Zhang. 
\newblock Two multiobjective problems for stochastic degenerate parabolic equations. 
\newblock{\em SIAM Journal on Control and Optimization}, 61 (2023), 2708--2735.

\bibitem{ZhangGaoyua}
S. Zhang, H. Gao and G. Yuan. 
\newblock New global Carleman estimates and null controllability for a stochastic Cahn-Hilliard type equation. 
\newblock{\em Journal of Differential Equations}, 430 (2025), 113203.

\bibitem{zhouZ12}
Z. Zhou. 
\newblock Observability estimate and null controllability for one-dimensional fourth order parabolic equation. 
\newblock{\em Taiwanese Journal of Mathematics}, 16 (2012), 1991--2017.

\end{thebibliography}
\end{document}